\documentclass[11pt,reqno]{amsart}
\usepackage{amssymb,mathrsfs,color,mathtools,empheq, verbatim, epstopdf}
\usepackage{pinlabel}
\mathtoolsset{showonlyrefs}
\usepackage{hyperref} 
\usepackage{enumerate, bbm}
\usepackage{tensor}
\usepackage{graphicx}
\usepackage{xcolor} % A package to add color.
\usepackage{slashed}
\usepackage{cite}
\usepackage[shortlabels]{enumitem}
\usepackage{geometry}
\numberwithin{equation}{section}

\newtheorem{thm}{Theorem}[section]
\newtheorem{cor}[thm]{Corollary}
\newtheorem{lem}[thm]{Lemma}

\theoremstyle{definition} 
\newtheorem{rem}[thm]{Remark}
\newtheorem{defn}[thm]{Definition}
\theoremstyle{remark}

\def\bR {\mathbb{R}}
\def\bS {\mathbb{S}}

\def\bZ {\mathbb{Z}}

\def\cL {\mathcal{L}}

\def\cZ {\mathcal{Z}}

\newcommand{\la}{\langle}
\newcommand{\ra}{\rangle}
\newcommand{\La}{\big\langle}
\newcommand{\Ra}{\big\rangle}

\newcommand{\tx}[1]{\mathrm{#1}}

\newcommand{\wt}[1]{\widetilde{#1}}

\newcommand{\dist}{\operatorname{dist}}

\newcommand{\vd}{\mathrm{d}}
\newcommand{\udr}{\,r\vd r}
\newcommand{\vD}{\mathrm{D}}

\newcommand{\uln}[1]{{\underline{ #1 }}}
\newcommand{\lin}{_{\textsc{l}}}

\definecolor{deepgreen}{cmyk}{1,0,1,0.5}

\newcommand{\A}{\mathcal{A}}

\newcommand{\E}{\mathcal{E}}

\newcommand{\LL}{\mathcal{L}}

\newcommand{\cS}{\mathcal{S}}

\newcommand{\N}{\mathbb{N}}
\newcommand{\R}{\mathbb{R}}

\newcommand{\Z}{\mathbb{Z}}

\newcommand{\al}{\alpha}
\newcommand{\be}{\beta}

\newcommand{\de}{\delta}

\newcommand{\om}{\omega}
\newcommand{\lam}{\lambda}
\newcommand{\te}{\theta}

\newcommand{\ta}{\tau}

\newcommand{\De}{\Delta}

\newcommand{\Lam}{\Lambda}

\newcommand{\p}{\partial}
\newcommand{\na}{\nabla}

\makeatletter

\newcommand{\Rmnum}[1]{\expandafter\@slowromancap\romannumeral #1@}
\makeatother

\newcommand{\ti}{\widetilde}

\newcommand{\U}{\underline}

\newcommand{\ang}[1]{\left\langle{#1}\right\rangle}
\newcommand{\abs}[1]{\left\lvert{#1}\right\rvert}

\newcommand{\EQ}[1]{\begin{equation}\begin{split} #1 \end{split}\end{equation}}

\newcommand{\Del}[1]{}

\newcommand{\mfor}{{\ \ \text{for} \ \ }}
\newcommand{\mas}{{\ \ \text{as} \ \ }}

\newcommand{\uD}{\operatorname{D}}

\definecolor{green}{rgb}{0,0.8,0}

\newcommand{\ud}{\mathrm{d}}

\newcommand{\eps}{\epsilon}

\newcommand{\bfd}{{\bf d}}

\newcommand{\bfp}{{\bf p}}
\newcommand{\bfq}{{\bf q}}

\newcommand{\bbS}{\mathbb S}

\newcommand{\calE}{\mathcal E}

\newcommand{\calL}{\mathcal L}

\newcommand{\calQ}{\mathcal Q}

\newcommand{\calS}{\mathcal S}

\newcommand{\calZ}{\mathcal Z}

\newcommand{\ULam}{\U{\Lam}}

\allowdisplaybreaks[4]

\begin{document}
	\parindent=0pt
	\title[ ]{Nonexistence of blow-up solutions with smooth radiation for energy-critical equivariant wave maps}
	\author[J. Jendrej]{Jacek Jendrej}
\address{Institut de Math\'{e}matiques de Jussieu, Sorbonne Universit\'{e}, Universit\'{e} Paris Cit\'{e},
4 place Jussieu, 75005 Paris, France.
%,\ \& Faculty of Applied Mathematics, AGH University of Krak\'ow, al. Adama Mickiewicza 30, 30-059 Krak\'ow, Poland.
}
\email{jendrej@imj-prg.fr}
\author[Y. Yin]{Yuchen Yin}
\address{School of Mathematical Sciences,
University of Science and Technology of China, Hefei 230026, Anhui, China.}
\email{yuchenyin@mail.ustc.edu.cn}
 \author[L. Zhao]{Lifeng Zhao}
\address{School of Mathematical Sciences,
University of Science and Technology of China, Hefei 230026, Anhui, China.}
\email{zhaolf@ustc.edu.cn}	
\begin{abstract}
   We study $k$-equivariant energy critical wave maps $\bR^{1+2} \to \bS^2$, for any equivariance degree $k\ge 2$.
   We prove that the radiation associated with any finite-energy blow-up solution cannot satisfy a certain regularity condition; in particular, it cannot be smooth.
   The assumption $k \geq 2$ is necessary, since for $k = 1$ such solutions are known to exist. The starting point of our analysis is the soliton resolution theorem. The key ingredient is a~novel application of the modulation method, in which we compare the effects of the radiation and inner bubbles to study the dynamic behavior of the widest bubble.

\end{abstract}
\maketitle
\section{Introduction}\label{Sec:1}

\subsection{Setting of the problem and main results}

Consider the Cauchy problem for a wave map $\Psi$ from $\R^{1+2}$ with Minkowski metric to the standard $2-$sphere $\bbS^2\subset \R^3$:
\begin{align}\label{WM-general}
    \p_{tt}\Psi-\De\Psi=(|\na \Psi|^2-|\p_t \Psi|^2)\Psi,\ \text{in}\ [0,\infty)\times\R^2,
\end{align}
with the initial data $(\Psi(0), \partial_t \Psi(0)) = (\Psi_0, \Psi_1)$.

For fixed $k\in \{1,2,\dots\}$, under $k-$equivariant symmetry, the solution $\vec \Psi$ takes the form
\begin{align}
    \Psi(t,re^{i\te})=(\sin u(t,r)\cos k\te,\sin u(t,r)\sin k\te, \cos u(t,r) )\in \bbS^2\subset\R^3.
\end{align}

The equation for $\vec u=(u,\p_t u)$ can be rewritten as follows:
\begin{align}
    \label{WM}
    \p_{tt}u(t,r)-\De u(t,r)+\frac{k^2}{r^2}\frac{\sin 2u(t,r)}{2}=0, \ (t,r)\in \R\times (0,\infty).
\end{align}

In this paper, we consider the case $k\geq 2$.

The energy is given by
\begin{align}
    E(u(t),\p_t u(t)):=2\pi\int_{0}^{\infty}\frac{1}{2}\left((\p_t u(t,r))^2+(\p_r u(t,r))^2+k^2\frac{\sin^2u(t,r)}{r^2}\right)r\ud r.
\end{align}

For $0\le r_1<r_2\le \infty$, we denote the localized nonlinear energy by
\begin{align}
E(\vec u(t);r_1,r_2)
:=
2\pi\int_{r_1}^{r_2}
\frac12
\left(
\p_tu(t,r)^2
+
(\partial_r u(t,r))^2
+
k^2\frac{\sin^2 u(t,r)}{r^2}
\right)r\,dr.
\end{align}
When $r_2=\infty$, we write
\[
E(\vec u;r_1):=E(\vec u;r_1,\infty).
\]

We will often write pairs of functions $\vec v=(v,\dot v)$, noting that the notation $\dot v$ will not, in general, refer to a time derivative of $v$ but rather just to the second component of $\vec v$.

 With this notation, the Cauchy problem for~\eqref{WM} can be rephrased as the Hamiltonian system 
\begin{equation} \label{eq:u-ham} 
\partial_t \vec u(t) =  J \circ \vD E( \vec u(t)),\qquad \vec u(T_0) = \vec u_0,
\end{equation} 
where 
\begin{equation} \label{eq:DE}
J = \begin{pmatrix} 0 &1 \\ -1 &0\end{pmatrix}, \quad \vD E( \vec u(t))  =  \begin{pmatrix}- \De u(t)  + k^2r^{-2}2^{-1}\sin(2u(t)) \\ \partial_t u(t) \end{pmatrix}.
\end{equation}
We define
\begin{align}
    \vD E_{\bfp}(u)=- \De u  + k^2r^{-2}2^{-1}\sin(2u).
\end{align}

The set of finite energy data is split into disjoint sectors, $\E_{\ell, m}$, which, for $\ell, m\in \bZ$, are defined by 
\begin{align}
\calE_{\ell, m}:= \big\{ (u_0, \dot u_0) \mid E(u_0, \dot u_0) < \infty, \quad \lim_{r \to 0} u_0(r) = \ell\pi, \quad \lim_{r \to \infty} u_0(r) = m \pi \big\}.
\end{align}

The sets $\calE_{\ell, m}$ are affine spaces, parallel to the linear space $\E := \E_{0, 0} = H \times L^2$, which we endow with the norm
\begin{align}
\| \vec u_0 \|_{\E}^2:=   \| \dot u_0 \|_{L^2}^2+ \| u_0\|_H^2  := \int_0^\infty \Big((\dot u_0(r))^2 + ( \p_r u_0(r))^2 + k^2 \frac{ (u_0(r))^2}{r^2} \Big) \, r \ud r.
\end{align}

The local energy norm is denoted by $\E(r_1,r_2)$. We adopt similar conventions as for $E$ regarding the omission of $r_2$, or both $r_1$ and $r_2$.

The stationary solutions to ~\eqref{WM} are described in terms of the $k$-equivariant harmonic map
\begin{align}
Q(r)=2\arctan(r^k).
\end{align}
More precisely, every finite energy stationary solution to~\eqref{WM}  takes the form $Q_{\mu, \sigma, m}(r) =  m \pi + \sigma Q(\frac{r}{\mu})$ for some  $\mu \in (0, \infty), \sigma \in \{0, -1, 1\}$ and $m \in \Z$. 
We write \(\vec Q:=(Q,0)\), and for \(\lambda>0\), $\vec Q_\lambda(r) := (Q_\lambda(r), 0)  := (Q(\lambda^{-1}r),0)$. These pairs are minimizers of the energy $E$  within the class $\E_{0, 1}$;  in fact, $E(\vec  Q_\lambda ) = 4  \pi k$. We also denote $\vec\pi := (\pi, 0)$.

\bigskip

The theory of wave maps has been extensively studied. The subcritical local Cauchy theory, based on bilinear and null-form estimates, was developed in the works of Klainerman and Machedon \cite{KM-1}, \cite{KM-2}, \cite{KM-3}, Klainerman and Selberg \cite{Kla-S-1}, \cite{Kla-S-2}, and Selberg \cite{Selberg}. In particular, these results yield local well-posedness for $\mathbb S^2$-valued wave maps \eqref{WM-general} in $H^s\times H^{s-1}$, $s>1$. The small-data critical theory, including the sphere target as a special case, was developed by Tao \cite{Tao-II} and Tataru \cite{Tataru01}, \cite{Tataru05}. For general compact targets, the large-data critical regularity theory was developed by Sterbenz and Tataru \cite{Ste-Ta}, \cite{Ste-Ta-2}. We also mention the large-energy theory for wave maps into hyperbolic targets developed by Tao \cite{Tao-III,Tao-IV,Tao-V,Tao-VI,Tao-VII} and by Krieger and Schlag \cite{KS-book}.

\bigskip

We now turn to the $k$-equivariant reduction \eqref{WM}. The study of equivariant wave maps was initiated by Shatah and Tahvildar-Zadeh \cite{Sha-Tah-1, Sha-Tah-2}.
For a solution blowing up at finite time $T_+$, we are concerned with its asymptotic description as $t \to T_+$.
For finite-energy equivariant solutions, blow-up can only occur at $r = 0$, and
it is a direct consequence of the finite speed of propagation
that $\vec u(t,r)$ converges strongly in $\E(R)$ for any $R>0$ to some $\vec u_0^*$ satisfying $E(\vec u_0^*) < \infty$.
We denote $\vec u^*(t)$ the solution of \eqref{WM}
such that $\vec u^*(T_+) = \vec u_0^*$.
We shall, by slight abuse of terminology, refer to both the initial data $\vec u_0^*:=\vec u^*(T_+)$ and the solution $\vec u^*(t)$ as the \emph{radiation}. We will specify which one is meant whenever ambiguity may arise.
Note that $\vec u$ and $\vec u^*$ coincide outside of the light cone with tip at $(T_+, 0)$.
In this sense, the radiation $\vec u^*$ captures the part of the solution which remains after blow up.

Once the radiation  $\vec u^*$ has been identified, the remaining part $\vec u(t)-\vec u^*(t)$ contains all the concentration responsible for the singularity. Soliton resolution gives a refined description of this concentrating component in terms of a superposition of a finite number of rescaled harmonic maps.
The first step towards soliton resolution was accomplished by Shatah and Tahvildar-Zadeh \cite{Sha-Tah-1, Sha-Tah-2}, who proved the non-concentration of energy at the self-similar scale; more precisely, for any $\alpha\in(0,1)$,
\begin{align}
\label{eq:rad-rho}
    E(\vec u(t);\alpha(T_+-t),(T_+-t))\to 0
    \quad \text{as } t\to T_+ .
\end{align}
Local convergence to a rescaled harmonic map inside the light cone emanating from the singularity was proved by Struwe \cite{Struwe-WM-03}. The proof of the full soliton resolution for ~\eqref{WM} has been established through several works; see \cite{Cote-WM-sequential} and \cite{Jia-Kenig-sequential} for sequential results, and \cite{DKMM} and \cite{JL-WM} for the proof in continuous time. See also \cite{Ary-NLH-SRC-Rad}, \cite{JL-HMHF-SRC},\cite{JL-NLW}, \cite{KK-CSS}, \cite{KK-CMS} for the proofs of soliton resolution in other models. For constructions of solutions to~\eqref{WM} exhibiting various types of non-trivial dynamical behavior, we refer to  \cite{GK-WM-KST}, \cite{hwang2026constructioninfinitetimebubble},  \cite{J-WM-two-bubble-construction}, \cite{jendrej2025concentricbubblesconcentratingfinite},  \cite{JLR-WM}, \cite{Jeong}, \cite{KST-WM}, \cite{RR-WM}, \cite{RS-sigma-model}.

%\Blue{We note that the radiation $\vec u^* $ can be determined without invoking soliton resolution.}
\bigskip

A natural problem is to ask how much freedom the radiation actually has. More precisely, one may ask whether an arbitrary element of the energy space can occur as the radiation of a nonlinear solution, or whether the nonlinear
dynamics impose additional restrictions on the possible radiation. Our result shows that sufficiently regular radiation is
incompatible with blow-up. In this sense, the theorem below can be viewed as a
rigidity statement for the radiation.
We prove that if a finite energy blow-up solution to ~\eqref{WM} exists, then the
corresponding radiation cannot be too regular. More precisely, we prove the following result. 

\begin{thm}\label{thm:main}
    For $k\geq 2$, there exists no finite energy solution  $ \vec u$  to \eqref{WM} blowing up at finite time $T_+<\infty$ with radiation $\vec u^*(t,r)$ satisfying 
    \begin{align}\label{eq:radiation-regularity-1}
        \Psi_0(re^{i\te})-(0,0,(-1)^{m})\in H^{3}(\R^2;\R^3),
    \end{align}
    \begin{align}\label{eq:radiation-regularity-2}
        \Psi_1(re^{i\te})\in H^{2}(\R^2;\R^3).
    \end{align}

    Here
    \begin{align}\label{eq:radiation-Psi}
         \Psi_0(re^{i\te})&:=(\sin u^*_0(r)\cos k\te,\sin u^*_0(r)\sin k\te,\cos u^*_0(r)),\\
         \Psi_1(re^{i\te})&:=(u_1^*(r)\cos u^*_0(r)\cos k\te,u_1^*(r)\cos u^*_0(r)\sin k\te,-u_1^*(r)\sin u^*_0(r))
    \end{align}
    with  $\vec u_0^*(r)=(u_0^*(r),u_1^*(r))=(u^*(T_+,r),\p_tu^*(T_+,r))$. The integer $m=\lim_{r\to \infty} \frac{u_0^*(r)}{\pi}$.
\end{thm}

\begin{rem}
   We make no claim that the regularity assumptions in Theorem~\ref{thm:main} are optimal. It would be interesting to determine the sharp regularity threshold.
\end{rem}

\begin{rem}
    The regularity assumption on the radiation in Theorem \ref{thm:main} is essential and is compatible with all currently known finite-time blow-up constructions for \eqref{WM}. In particular,  Rapha\"el and Rodnianski \cite{RR-WM} constructed, for any $k\geq 2$, an open set of initial data in $H^2$ that leads to blow-up. The radiation associated with these blow up solutions is no better than energy class. In addition, finite-time blow-up solutions with two concentrating bubbles \cite{jendrej2025concentricbubblesconcentratingfinite} and, more generally, with an arbitrary number of concentric concentrating bubbles \cite{krieger2026longfinitetimebubble} have been constructed in the case $k=2$. None of these examples satisfies the regularity assumption imposed in Theorem \ref{thm:main}.
\end{rem}

\begin{rem}
    The requirement of the equivariance class $k\geq 2$ is necessary. In fact, the analog of our result is false for $k=1$. Indeed, in
    \cite{JLR-WM}, the first author, Lawrie, and Rodriguez constructed
    blow-up solutions with radiation of the form  $ \vec u_0^*(r)=(q\chi(r)r^{\nu},0)$,  where $q<0$ and $\nu>\frac{9}{2}$ can be chosen arbitrarily. The solutions constructed in \cite{Jeong} also have arbitrarily smooth radiation.
    \end{rem}

    \begin{rem}
        The analog of $\vec u^*$ for solutions defined for all $t \in [0,\infty)$
        was studied in \cite{Cote-WM-sequential}, \cite{CKLS-WM-1}, \cite{CKLS-WM-2},\cite{Jia-Kenig-sequential}.
It is given by the unique solution of the linear wave equation which approximates $u$ outside of all shifted future light cones (up to an error whose energy converges to $0$).

    For global solutions, the analog of Theorem~\ref{thm:main} may fail. Indeed, in \cite{hwang2026constructioninfinitetimebubble} Hwang and Kim constructed infinite time bubble towers for ~\eqref{WM} in $k\geq 3$ without radiation. Strictly speaking, a single stationary soliton also provides a counterexample. 
\end{rem}

\begin{rem}
    Similar nonexistence results were also obtained for some parabolic equations. In \cite{KimM-NLH-HMHF-radial} Kim and Merle proved the nonexistence of finite energy blow up solutions for radial focusing energy critical heat equations in dimensions $D\geq 7$  and $k$-equivariant harmonic map heat flows with $k\geq 3$. See also \cite{kim2026rigidityresultsmultibubbledynamics} for a nonradial analog for focusing energy critical heat equations. 
\end{rem}

\subsection{Strategy of the proof}

We now give an informal outline of the proof.

\bigskip

We argue by contradiction. If such a solution exists, we reduce it to a solution $\vec u$ with small compactly supported radiation using finite speed of propagation and use soliton resolution to describe the concentrating part. Theorem~\ref{thm:src-blowup} implies that the solution can be decomposed into the form
\begin{align}\label{eq:decomposition-u-thm}
    \vec u(t) = m_\De  \vec \pi + \sum_{j =1}^N \iota_j(\vec Q_{\lam_j(t)} - \vec \pi) + \vec u^*(t) + \vec g(t) , 
\end{align}
such that 
\begin{align}
      \| \vec g(t)\|_{\E} + \sum_{j =1}^{N-1} \frac{\lam_{j}(t)}{\lam_{j+1}(t)}+\frac{\lam_N(t)}{T_+-t}  \to 0 \mas t \to T_+. 
\end{align}
Here $\vec u^*(t)$ is the radiation with initial data $\vec u_0^*$ at time $T_+$. We will study the dynamical behavior of the scale of the widest bubble, $\lam_N$.

\bigskip

There are two mechanisms governing the evolution of $\lam_N(t)$. The first one comes from the interaction with the inner bubbles, while the
second one comes from the interaction with the radiation.  Let us first describe the two mechanisms separately. If the widest bubble
is affected only by the inner bubbles,
 then the argument used in the proof of
Lemma 5.8 in ~\cite{JL-WM} would show that the cumulative effect of the inner
bubbles is too small to drive $\lam_N(t)$ to zero. More precisely, after
integration in time, this contribution is negligible compared with
$\lambda_N(t)$, which would force $\lambda_N(t)$ to remain comparable to a
positive constant, contradicting blow-up. On the other hand, if the widest bubble is affected only by the radiation, then the following energy expansion suggests the same conclusion:
\begin{align}
    &E(\vec Q)+E(\vec u^*)=
    E(\iota\vec Q_{\lam(t)}+\vec u^*(t))+\La \vD E(\iota\vec Q_{\lam(t)}+\vec u^*(t)),\vec g(t)\Ra\\
    &+\frac{1}{2}\La\vD^2E(\iota\vec Q_{\lam(t)}+\vec u^*(t) )\vec g(t),\vec g(t)\Ra+O(\|\vec g(t)\|_{\E}^3).
\end{align}

Following the idea from \cite{J-NLW-speed}, the term $\La \vD E(\iota\vec Q_{\lam(t)}+\vec u^*(t)),\vec g(t)\Ra $  can be decomposed into a small term and an almost constant term, and hence can be ignored. The term $\frac{1}{2}\La\vD^2E(\iota\vec Q_{\lam(t)}+\vec u^*(t) )\vec g(t),\vec g(t)\Ra$ is coercive under some suitable condition. Then formally we have
\begin{align}
    |\lam'|^2\lesssim\|\vec g\|_{\E}^2\lesssim |E(Q)+E(u^*)-E(\iota\vec Q_{\lam(t)}+\vec u^*(t))|\lesssim\lam^{2}
\end{align}
under our assumption on $\vec u^*$. This also implies that $\lam_N$ cannot reach zero in finite time, a contradiction.  

\bigskip

The main difficulty is that, in the actual problem, both effects are present
simultaneously. Our strategy is to split the time interval into two regimes. In
the first regime, the dynamics of $\lam_N(t)$ is dominated by the interaction
with the inner bubbles, and $\lam_N(t)$ cannot change
significantly. In the second regime, the contribution of the radiation dominates, which 
allows an Osgood-type argument to show that $\lam_N(t)$  cannot reach zero in finite time. Combining the two estimates rules out the possible
mixed dynamics and yields the desired contradiction. 

\bigskip

A key technical point is to quantify the interaction between the 
bubbles and the radiation. This is where the $k$-equivariant structure is used:
the assumed regularity of the radiation implies suitable decay estimates near
the origin, which allow us to control the radiation--bubble interaction in the
energy expansion.

\bigskip

Our paper is organized as follows. In Section \ref{Sec:2}, we recall some basic properties of wave maps and prove the required properties of radiation. In Section \ref{Sec:3}, we use the modulation method to study the dynamical behavior of the scaling parameters. In Section \ref{Sec:4}, we use energy conservation to derive the key energy estimates. We treat each term in the energy expansion carefully. In Section \ref{sec:proof}, we combine the results in these sections and derive the contradiction.

\subsection{Notation}

Given a function $u(r)$ and $\lam>0$, we define
\begin{align}
    u_{\lam}(r)=u\left(\frac{r}{\lam}\right),\quad u_{\U\lam}(r)=\frac{1}{\lam}u\left(\frac{r}{\lam}\right).
\end{align}

We denote the infinitesimal generators of these scalings by
\begin{align}
    \Lam=r\p_r,\quad \U\Lam=r\p_r+1.
\end{align}

Throughout this paper $\chi$ always denotes a smooth cutoff function  such that  $\chi|_{\{r\leq 1\}}=1$, $\chi|_{\{r\geq 2\}}=0$, $0\leq \chi\leq 1$. We define $\chi_R(r)=\chi(\frac{r}{R})$ for $R>0$.

We define
\begin{align}
    f(u)=\frac{1}{2}\sin2u.
\end{align}

\noindent{\bf Acknowledgment.}  J. Jendrej was supported by the ERC project INSOLIT (No. 101117126). L. Zhao was supported by National Natural Science Foundation of China (No. 12271497 and No. 12341102).

\section{Preliminaries}\label{Sec:2}

\subsection{Cauchy theory}

The local Cauchy theory for ~\eqref{WM-general}  in space $H^s\times H^{s-1}$ with $s>1$ has been studied in
\cite{KM-1}, \cite{KM-2}, \cite{KM-3}, \cite{Kla-S-1}, \cite{Kla-S-2}  and \cite{Selberg}. We shall use the following version, as recorded
in \cite{DJKM-WM}.
\begin{lem}[\rm{\cite[Theorem 2.1]{DJKM-WM}}]\label{lem:Cauchy-general}
    Let $s > 1$ and $\frac{1}{2} < b < \min\{s - \frac{1}{2}, 1\}$. Suppose that $(\Psi_0, \Psi_1) \in \dot{H}^s \times H^{s-1}$ and that $\Psi_0$ equals a constant $\Psi_\infty \in \bbS^2$ for large $x$. Then for $T = T(\| (\Psi_0 - \Psi_\infty, \Psi_1) \|_{H^s \times H^{s-1}}) > 0$ sufficiently small, there exists a unique solution $\vec \Psi$ to equation ~\eqref{WM-general} with initial data $(\Psi_0, \Psi_1)$ on $\R^2 \times (-T, T)$ in the sense of distributions, which satisfies the following properties
\begin{enumerate}
    \item[(1)] $(\Psi - \Psi_\infty , \p_t \Psi)\in C(I, H^s \times H^{s-1})$;
    \item[(2)] there exists $\bar{\Psi} \in L^2(\R^3)$ with $\bar{\Psi}|_{\R^2 \times I} \equiv \Psi - \Psi_\infty$ and $\nabla_{x,t} \bar{\Psi} \in X^{s-1,b}$,
\end{enumerate}
where $I = (-T, T)$.
\end{lem}

We recall the Cauchy theory for ~\eqref{WM} proved by Shatah and Tahvildar-Zadeh in \cite{Sha-Tah-1}, \cite{Sha-Tah-2}.
%\emph{\cite[Theorem 1.1]{Sha-Tah-2},\cite[Theorem 8.1]{Sha-Str-1}\,  \cite{Sha-Tah-1}}
\begin{lem}[Local well-posedness,  \rm{\cite[Theorem 1.1]{Sha-Tah-2},\cite[Theorem 8.1]{Sha-Str-1}, \cite{Sha-Tah-1}} ]\label{lem:lwp-1}  Let $\ell, m \in \Z$ and let $\vec u_0 \in \E_{\ell, m}$. Then, there exists a maximal time interval of existence $(T_-, T_+) = I_{\max}(\vec u_0) \ni 0$ on which~\eqref{WM} admits a unique solution $\vec u(t)$ in the space $ C^0(I_{\max}; \E_{\ell, m})$ with $\vec u(0) = \vec u_0$. 

In fact, there exists $\eps_0 >0$ with the following property. Let $\vec u_0 \in \E_{\ell, m}$, $\tau>0$ and  suppose the solution $\vec u(t)$ to~\eqref{WM} with data $\vec u(0) = \vec u_0$ is defined on the interval $[0, \tau)$, i.e., in $C^0([0, \tau); \E_{\ell, m})$. Suppose that there exists a time $0\le t< \tau$ and a number $R> \tau - t$ such that, 
\begin{align}
E( \vec u(t); 0, R) < \eps_0. 
\end{align}
Then,  $T_+( \vec u) > \tau$. 
\end{lem} 

We also recall the small data theory for $~\eqref{WM}$, which can be proved by a standard argument based on the contraction mapping principle, see \cite{CRKM-YM-WM} for an example. We shall use the version presented in \cite{JL-WM}.

\begin{lem}[Cauchy theory in $\E$, \rm{\cite[Lemma 2.8]{JL-WM}}]  \label{lem:Cauchy}  There exist functions $\de_0, C_0: [0, \infty) \to (0, \infty)$ with the following properties. Let $A\ge 0$ and let $\vec u_0 = (u_0, u_1) \in \E$ with $\|\vec u_0 \|_{\E} \le A$. Let $I \ni 0$ be an open interval such that 
\begin{align}
\| S\lin(t)  \vec u_0 \|_{\calS(I)}   = \de \le  \de_0(A), 
\end{align}
Here $S\lin(t)$ is the linear operator related to the linear equation
\begin{align}
\label{eq:lin} 
\p_t^2 v - \De v + \frac{k^2}{r^2} v = 0,
\end{align}
and the Strichartz norm $\cS(I)$ is defined as
\begin{align}
    \|w\|_{\cS(I)}:= &\left\| r^{-\frac{3}{5}} w \right\|_{L^5_{t,r}(I)}
+
\left\| r^{-\frac{2}{3}} w \right\|_{L^3_t L^6_r(I)},
\end{align}
and we use the convention $\|\vec w\|_{\cS(I)} :=\|w_0\|_{\cS(I)} $ for $\vec w=(w_0,w_1) $.

Then there exists a unique solution
$\vec u(t)=(u(t),\partial_t u(t))$ to~\eqref{WM}
with initial data $\vec u(0)=\vec u_0$, satisfying
\begin{align}
\vec u\in C^0(I;\E),  u\in \cS(I).
\end{align} 
Moreover, $\vec u(t)$ satisfies the bounds $\| u \|_{\cS(I)} \le C(A) \de$, and $ \| \vec u \|_{L^\infty_t(I; \E)} \le C(A)$. To each solution $\vec u(t)$ to~\eqref{WM} we can associate a maximal interval of existence $I_{\max}(\vec u)$ such that for each compact subinterval $I' \subset I_{\max}$ we have $\| u \|_{\cS(I')} < \infty$. 

Finally, there exists $\eps_0$ small enough so that for all $\vec u_0 \in \E$ satisfying $E(\vec u_0) < \eps_0$, the solution $\vec u(t)$ given above is defined globally in time and satisfies the bound 
\begin{align}
\sup_{t \in \R}\| \vec u(t) \|_{\E} + \| u \|_{\calS(\R)} \lesssim \|\vec u_0\|_{\E}.
\end{align}

\end{lem}

\subsection{Soliton resolution}

\begin{thm}[Soliton resolution for blow-up solutions, \rm{\cite[Theorem 1]{JL-WM}}] \label{thm:src-blowup}  Let $k \in \N$, let $\ell, m \in \Z$,  and let $\vec u(t)$ be a finite energy solution to~\eqref{WM} with initial data $\vec u(0) = \vec u_0 \in \E_{\ell, m}$,  defined on its maximal forward interval of existence $[0,T_+)$. If $T_+ < \infty$, there exists a time $T_0< T_+$,  integers $m_\De,m_{\infty}$, a mapping $\vec u_0^*\in \E_{0, m_{\infty}}$, an integer $N \ge 1$, continuous functions $\lam_1(t), \dots,  \lam_N(t) \in C^0([T_0, T_+))$, signs $\iota_1, \dots, \iota_N \in \{-1, 1\}$, and $\vec g(t) \in \E$ defined by 
\begin{align}\label{eq:decomposition-u-thm}
    \vec u(t) = m_\De  \vec \pi + \sum_{j =1}^N \iota_j(\vec Q_{\lam_j(t)} - \vec \pi) + \vec u^*_0 + \vec g(t) , 
\end{align}
such that 
\begin{align}
      \| \vec g(t)\|_{\E} + \sum_{j =1}^{N} \frac{\lam_{j}(t)}{\lam_{j+1}(t)}  \to 0 \mas t \to T_+, 
\end{align}
where above we use the convention that $\lam_{N+1}(t) = T_+-t$. 

Analogous statements hold for the backwards-in-time evolution. 
\end{thm} 

An immediate result of this theorem is that for $\vec u$ with the decomposition ~\eqref{eq:decomposition-u-thm}, we have
\begin{align}\label{eq:energy-decoupling}
    E(\vec u)=NE(\vec Q)+E(\vec u^*),
\end{align}
which is the starting point of Section~\ref{Sec:4}

\subsection{Modification of the decomposition}

 For any finite energy blow-up solution $\vec v$ to ~\eqref{WM}, Theorem~\ref{thm:src-blowup} yields a decomposition of the following form:
    \begin{align}
    \vec v(t) = m_\De  \vec \pi + \sum_{j =1}^N \iota_j(\vec Q_{\mu_j(t)} - \vec \pi) + \vec v^*_0 + \vec h(t) ,
\end{align}
    such that 
\begin{align}
      \| \vec h(t)\|_{\E} + \sum_{j =1}^{N} \frac{\mu_{j}(t)}{\mu_{j+1}(t)}  \to 0 \mas t \to T_+.
\end{align}
We denote by $\vec v^*(t,r)$  the solution to ~\eqref{WM} with initial data $\vec v_0^*(r)$ at time $T_+$ .
 
\begin{lem}\label{lem:small-radiation-valid}
     Let $\de>0$ be arbitrary. Consider $\vec v$ and the decomposition given above. If $v_0^*(r)$ satisfies ~\eqref{eq:radiation-regularity-1} and ~\eqref{eq:radiation-regularity-2} with $\vec u_0^*$ replaced by $\vec v_0^*$, then there exists a finite energy blow-up solution $\vec u$ to ~\eqref{WM} defined on $(T_1,T_+)$ for some $T_1$ sufficiently close to $T_+$ admitting a decomposition
    \begin{align}
        \vec u(t) =  \sum_{j =1}^N \iota_j\vec Q_{\mu_j(t)}  + \vec u^*_0 + \vec {\ti h}(t) ,
    \end{align}
    such that

    \begin{enumerate}[label=(\roman*)]
        \item \begin{align}\label{eq:ortho-lem}
        \| \vec {\ti h}(t)\|_{\E} + \sum_{j =1}^{N} \frac{\mu_{j}(t)}{\mu_{j+1}(t)}  \to 0 \mas t \to T_+,
    \end{align}
        \item $\vec u^*_0\in \E$ is compactly supported and 
    \begin{align}\label{eq:radiation-small-lem}
        \sup_{t\in(T_1,T_+]}\|\vec u^*(t)\|_{\E}<C\de,
    \end{align}
    where $\vec u^*(t,r)=(u^*(t,r),\p_t u^*(t,r))$ is the solution to ~\eqref{WM} with initial data $\vec u_0^*(r)$ at time $T_+$ and $C$ depends only  on the equivariance degree $k$.
        \item For $\vec u^*(t,r)$ defined above we have for $t\in (T_1,T_+]$ and $k=2$,
        \begin{align}\label{eq:radiation-decay-k-2}
        |u^*(t,r)|\lesssim\min\{r^2\left(1+|\log r|^{\frac{1}{2}}\right),1\},
    \end{align}
    for $k\geq 3$,
        \begin{align}\label{eq:radiation-decay-k-3}
        |u^*(t,r)|\lesssim\min\{r^2,1\}.
    \end{align}
    Here the implicit constants are independent of $t$.
        \item There exists a constant $\rho>0$ such that, for all $t\in(T_1,T_+]$,
        \begin{align}\label{eq:radiation-decay-ut-outside}
            \partial_tu^*(t,r)=0, \ \mfor \ r\geq 8\rho,
        \end{align}
        and
        \begin{align}\label{eq:radiation-decay-ut-local}
             |\p_tu^*(t,r)|\lesssim\min\{r,1\},\ \mfor \ r\leq\rho,
        \end{align}
        where the implicit constants are independent of $t$.

    \end{enumerate}

\end{lem}

\begin{rem}
    We note that $\vec u^*$ in this lemma is not the radiation as defined before. However, we still call it radiation for convenience.
\end{rem}

\begin{proof}
    Let $\vec v_1(t,r)=\vec v(t,r)-m_{\De}\vec \pi$. We choose $T_1$ sufficiently close to $T_+$ such that 
    \begin{align}
        \|\vec v_1(T_1)\|_{H\times L^2(\frac{1}{100}(T_+-T_1),100(T_+-T_1))}<\frac{1}{2}\de,
    \end{align}
    \begin{align}
        \sup_{t\in[T_1,T_+]}\|\vec v^*(t,r)\|_{\E(0,100(T_+-T_1))}<\de,
    \end{align}
    and take $R=5(T_+-T_1)$. The existence of such $T_1$ is a consequence of Theorem~\ref{thm:src-blowup}. We consider 
    $\chi_R\vec v_1(T_1)$ and propagate it forward through ~\eqref{WM}. We denote the new solution by $\vec u(t,r)$. By finite speed of propagation we know that $\vec u(t,r)=\vec v_1(t,r)$ in the region $\{r\leq R-t+T_1\}$. Since $ R-T_++T_1>0 $, we know that $\vec u(t,r)$ blows up at $t=T_+$. Now we consider the radiation $\vec u^*_0(r)$ of $\vec u$. 
    Finite speed of propagation yields that $\vec u^*_0(r)=\vec v^*_0(r)$ in the region $\{r\leq R-T_++T_1\}$, and $\vec u^*_0(r)=\vec 0$ in the region $\{r\geq 2R+T_+-T_1\}$. As for the region $\{ R-T_++T_1< r<2R+T_+-T_1\}$, applying Lemma~\ref{lem:Cauchy} to $ (1-\chi_{\frac{1}{50}(T_+-T_1)})\chi_R\vec v_1(T_1,r) $  we get
    \begin{align}
        \|\vec u^*_0\|_{\E(R-T_++T_1,\,2R+T_+-T_1)}
\lesssim
\|\vec v_1(T_1)\|_{\E(\frac{1}{100}(T_+-T_1),\,100(T_+-T_1))}
\leq \de.
    \end{align}
    Therefore the radiation  satisfies
    \begin{align}
        \|\vec u^*_0\|_{\E}\lesssim \de,
    \end{align}
    and the proof of ~\eqref{eq:ortho-lem} and ~\eqref{eq:radiation-small-lem} is completed by considering $\vec u+\sum_{j=1}^N\iota_j \vec \pi$ and using Lemma~\ref{lem:Cauchy} for $\vec u^*_0(r)$. 

    Now we prove ~\eqref{eq:radiation-decay-k-2} and ~\eqref{eq:radiation-decay-k-3}. We first consider $u^*$. By radial Sobolev embedding we have $|u^*(t,r)|\lesssim 1$. We only have to consider the behavior of $u^*(t,r)$ near $r=0$. Since $\vec u^*(t,r)=\vec v^*(t,r)$ in $\{(r,t):T_1<t<T_+,\ r\leq t+R+T_1-2T_+\}$, using Hardy's inequality,  we know that  ~\eqref{eq:radiation-regularity-1} and ~\eqref{eq:radiation-regularity-2} hold for $\vec u^*_0(r)$ replaced by $\chi_{\frac{R-T_++T_1}{3}}(r)\vec u^*_0(r)$. 
    
    Since $\vec u^*(t,r)=\vec v^*(t,r)$ in $\{(r,t):T_1<t<T_+,\ r\leq t+R+T_1-2T_+\}$, by considering the solution with initial data $\chi_{\frac{R-T_++T_1}{3}}(r)\vec u^*_0(r)$ at time $T_+$, and using Lemma~\ref{lem:Cauchy-general} for $\vec \Psi$ related to $\chi_{\frac{R-T_++T_1}{3}}(r)\vec u^*_0(r)$, we may assume that $\vec u^*(t,r)$ is supported in $\{r\leq 1\}$, and for $t$ sufficiently close to $T_+$,  $\vec u^*(t,r)$ satisfies  ~\eqref{eq:radiation-regularity-1} and ~\eqref{eq:radiation-regularity-2} uniformly in $t$  with $u^*_0(r)$ replaced by $u^*(t,r)$. 

    Let $f(t,r)=\sin u^*(t,r)$, then we have $f(t,r)e^{ik\te}\in H^3(\R^2)$. We claim that
    \begin{align}\label{eq:f-k-1}
        \int_0^1\frac{|f(t,r)|^2}{r^3}\ud r+\int_0^1\left|\p_r\left(\frac{f(t,r)}{r^2}\right)\right|^2r\ud r\lesssim1, \ \mfor\ k\geq 2,
    \end{align}
    \begin{align}\label{eq:f-k-2}
        \int_0^1\frac{|f(t,r)|^2}{r^5}\ud r\lesssim1, \ \mfor\ k\geq 3.
    \end{align}

    The proof of these estimates will be shown in Appendix~\ref{sec:appendix-f}.     

    For $k=2$ and $0< r_0< 1$ we estimate
    \begin{align}
        \frac{|f(t,r_0)|}{r_0^2}\leq& \frac{|f(t,1)|}{1^2}+\int_{r_0}^1\left|\p_r\left(\frac{f(t,r)}{r^2}\right)\right|\ud r\\
        \lesssim&\  1+\left(\int_{r_0}^1\left|\p_r\left(\frac{f(t,r)}{r^2}\right)\right|^2r\ud r\right)^{\frac{1}{2}}\left(\int_{r_0}^1\frac{1}{r}\right)^{\frac{1}{2}}\\
        \lesssim&1+|\log r_0|^{\frac{1}{2}}.
    \end{align}
    
    For $k\geq 3$, we have for $0<r_1<r_2<1$,
    \begin{align}
        \left|\left(\frac{f(t,r_1)}{r_1^2}\right)^2-\left(\frac{f(t,r_2)}{r_2^2}\right)^2\right|\lesssim &\int_{r_1}^{r_2}\frac{|f(t,r)|}{r^2}\left|\p_r\left(\frac{f(t,r)}{r^2}\right)\right|\ud r \\
        \lesssim &\left(\int_{r_1}^{r_2}\frac{|f(t,r)|^2}{r^5}\ud r\right)^{\frac{1}{2}}\left(\int_{r_1}^{r_2}\left|\p_r\left(\frac{f(t,r)}{r^2}\right)\right|^2r\ud r\right)^{\frac{1}{2}},
    \end{align}
    which combined with ~\eqref{eq:f-k-1} and ~\eqref{eq:f-k-2} yields $\lim_{r\to 0}\frac{f(t,r)}{r^2}=0$. Therefore we have for $0<r_0<1$,
    \begin{align}
        \left|\frac{f(t,r_0)}{r_0^2}\right|^2\lesssim &\int_{0}^{r_0}\frac{|f(t,r)|}{r^2}\left|\p_r\left(\frac{f(t,r)}{r^2}\right)\right|\ud r \\
        \lesssim &\left(\int_{0}^{r_0}\frac{|f(t,r)|^2}{r^5}\ud r\right)^{\frac{1}{2}}\left(\int_{0}^{r_0}\left|\p_r\left(\frac{f(t,r)}{r^2}\right)\right|^2r\ud r\right)^{\frac{1}{2}}\\
        \lesssim&\ 1.
    \end{align}

    Recall that $\sup_{t\in (T_1,T_+]}\|u^*(t,r)\|_{L^{\infty}} < C\de$. Taking $\de>0$ sufficiently small, we obtain $\sin u^*(t,r)\sim u^*(t,r)$, which implies the estimates for $u^*(t,r)$ in ~\eqref{eq:radiation-decay-k-2} and ~\eqref{eq:radiation-decay-k-3}.  The estimates for $\p_t u^*(t,r)$ are similar, we choose $\rho=\frac{3}{2}(T_+-T_1)$ and omit the details.

\end{proof}

In the rest of this paper we always assume that $\vec u^*\in \E$ is compactly supported, and $\|\vec u^*\|_{\E}$ is small enough so that Lemma~\ref{lem:D2E-coercive}  holds for $v$ replaced by $u^*$.

\subsection{Pointwise interaction estimates}

We introduce some functions which will be frequently used in  Sections~\ref{Sec:3} and~\ref{Sec:4}.

We define
\begin{align}
    F_0(r)=\begin{cases}
        \min\{r^2(1+|\log r|^{\frac{1}{2}}),1\}, &\mfor  k=2,\\ 
        \min\{r^2,1\} &\mfor k\geq 3,
    \end{cases}
\end{align}
and
\begin{align}
    \ti F_0(r)=\min\{r,1\},\ F_j(t,r)=\min\left\{\left(\frac{\lam_j(t)}{r}\right)^k, \left(\frac{r}{\lam_j(t)}\right)^k\right\},\ 1\leq j\leq N.
\end{align}

Then, using  ~\eqref{eq:radiation-decay-k-2} and ~\eqref{eq:radiation-decay-k-3}  we have
\begin{align}
    |u^*(t,r)|\lesssim F_0(r).
\end{align}
~\eqref{eq:radiation-decay-ut-local} implies
\begin{align}
    |\p_tu^*(t,r)|\lesssim \ti F_0(r), \ \mfor\ 0\leq r\leq \rho.
\end{align}

We first record a basic inequality.

\begin{lem}
    For any $M\in \N$ and $x_0,\dots,x_M\in \R$, 
\begin{align}\label{eq:decouple}
    \left| \sin\left(2\sum_{i=0}^M x_i\right) - \sum_{i=0}^M \sin(2x_i) \right| \lesssim \sum_{0 \le i < j \le M} \dist(x_i,\pi\Z)\dist(x_j,\pi \Z).
\end{align}
\end{lem}

\begin{proof}
    For each $j\in{\{0,1,\dots,M\}}$ we take $m_j\in \Z$ such that $y_j=x_j-m_j\pi\in (-\frac{\pi}{2},\frac{\pi}{2}]$. Then we only have to prove that
\begin{align}\label{eq:decouple-1}
    \left| \sin\left(2\sum_{i=0}^M y_i\right) - \sum_{i=0}^M \sin(2y_i) \right| \lesssim \sum_{0 \le i < j \le M} |y_i||y_j|.
\end{align}

We estimate
\begin{align}
     &\left|\sin\left(2\sum_{i=0}^M y_i\right) - \sum_{i=0}^M \sin(2y_i)\right|\\ =& \left|\sum_{m=1}^M\left[ \sin\left(2\sum_{i=0}^m y_i\right)- \sin\left(2\sum_{i=0}^{m-1} y_i\right)- \sin(2y_m)\right]\right|\\
     =& \left|\sum_{m=1}^M\left[ \sin\left(2\sum_{i=0}^{m-1} y_i\right)\left(\cos(2y_m)-1\right)+\sin(2y_m)\left(\cos\left(2\sum_{i=0}^{m-1} y_i\right)-1\right)\right]\right|\\
     \lesssim&  \sum_{m=1}^M\sum_{i=0}^{m-1}\left|\sin(2y_i)\right||\sin(y_m)|^2+ \sum_{m=1}^M|\sin(2y_m)|\left(\sum_{i=0}^{m-1}\left|\sin(y_i)\right|\right)^2\\
     \lesssim& \sum_{0 \le i < j \le M} |y_i||y_j|.
\end{align}
\end{proof}

\begin{cor}
    We have
    \begin{align}\label{eq:sin-decouple}
    &\left| \sin\left(2\sum_{i=1}^N \iota_iQ_{\lam_i(t)}(r)\right) - \sum_{i=1}^N \iota_i\sin(2Q_{\lam_i(t)}(r)) \right| 
    \lesssim \sum_{1 \le i < j \le N} F_i(t,r)F_j(t,r).
\end{align}

\EQ{\label{eq:sin-decouple-1}
    &\left| \sin\left(2u^*(t,r)+2\sum_{i=1}^N \iota_iQ_{\lam_i(t)}(r)\right) -\sin 2 u^*(t,r)- \sum_{i=1}^N \iota_i\sin(2Q_{\lam_i(t)}(r)) \right| \\
    \lesssim & F_0(r)\sum_{j=1}^NF_j(t,r)+ \sum_{1 \le i < j \le N} F_i(t,r)F_j(t,r).
}

Here the implicit constants are independent of $t$ and $r$.
\end{cor}

\begin{proof}
    We take $x_j=\iota_jQ_{\lam_j(t)}(r)$ when $1\leq j\leq N$, $x_0=0$ for ~\eqref{eq:sin-decouple} and $x_0=u^*(t,r)$ for ~\eqref{eq:sin-decouple-1}. Then we have
\begin{align}
    \dist(\iota_jQ_{\lam_j(t)},\pi\Z)\lesssim\min\left\{\left(\frac{\lam_j(t)}{r}\right)^k, \left(\frac{r}{\lam_j(t)}\right)^k\right\}=F_j(t,r)
\end{align}
and
\begin{align}
    \dist(u^*(t,r),\pi\Z)\lesssim|u^*(t,r)|\lesssim F_0(r).
\end{align}
\end{proof}

The estimates for integrals of products involving these  $F_j$'s are computed in Appendix~\ref{sec:appendix-Fj}.

\section{Estimates on modulation parameters}\label{Sec:3}

In this section we study the solution to ~\eqref{WM} using the modulation method.

We first recall the localized virial operator defined in \cite{JL-WM}.

\begin{lem}[\rm{\cite[Lemma 4.13]{JL-WM}}]  
\label{lem:q} 
For any $c > 0$ and $R > 1$ there exists a function $q = q_{c, R} \in C^4((0, \infty))$
having the following properties:
\begin{enumerate}[(P1)]
\item $q(r) = \frac 12 r^2$ for all $r \in [R^{-1}, R]$,
\item there exists $\wt R > 0$ (depending on $c$ and $R$)
such that $q(r) = \tx{const}$ for $r \geq \wt R$ and $q(r) = \tx{const}$ for $r \leq \wt R^{-1}$,
\item $|q'(r)| \lesssim r$ and $|q''(r)| \lesssim 1$ for all $r > 0$,
with constants independent of $c$ and $R$,
\item $q''(r) \geq -c$ and $\frac 1r q'(r) \geq -c$ for all $r > 0$,
\item $\big|\big(\frac{\vd^2}{\vd r^2} + \frac 1r\frac{\vd}{\vd r}\big)^2q(r)\big| \leq cr^{-2}$ for all $r > 0$,
\item $\big|\big(\frac{q'(r)}{r}\big)'\big| \leq cr^{-1}$ for all $r > 0$.
\end{enumerate}
\end{lem}

\begin{defn}[Localized virial operator]
For each $\lam>0$ we set
\begin{align}
A(\lambda)g(r) &:= q'\big(\frac{r}{\lambda}\big)\cdot \p_r g(r), \label{eq:opA}\\
\uln A(\lambda)g(r) &:=\big(\frac{1}{2\lambda}q''\big(\frac{r}{\lambda}\big) + \frac{1}{2r}q'\big(\frac{r}{\lambda}\big)\big)g(r) + q'\big(\frac{r}{\lambda}\big)\cdot\p_r g(r). \label{eq:opA0}
\end{align}
These operators depend on $c$ and $R$ as in Lemma~\ref{lem:q}. 
\end{defn}

We also record some properties of $A(\lam)$ and $\uln A(\lam)$ which will be used in the proof of Lemma~\ref{lem:mod}. 

\begin{lem}[Localized virial estimates]  \label{lem:A} 
For any $c_0>0$ there exist $c_1, R_1>0$, so that for all $c, R$ as in Lemma~\ref{lem:q} with $c< c_1$, $R> R_1$ the operators $A(\lambda)$ and $\U A(\lambda)$ defined in~\eqref{eq:opA} and~\eqref{eq:opA0} have the following properties:

  \begin{itemize}[leftmargin=0.5cm]
    \item the families $\{A(\lambda): \lam > 0\}$, $\{\U A(\lambda): \lam> 0\}$, $\{\lambda\partial_\lambda A(\lambda): \lam > 0\}$
      and $\{\lambda\partial_\lambda \U A(\lambda): \lambda > 0\}$ are bounded in $\mathscr{L}(H; L^2)$, with the bound depending only on the choice of the function $q(r)$,
    \item  
   Let $g_1,g_2\in H $. Then, for all $\lambda > 0$, 
      \begin{multline}  \label{eq:A-by-parts}
      \Big| \ang{ A(\lambda)g_1\mid  \frac{k^2}{r^2}\big(f(g_1 + g_2) - f(g_1) - f'(g_1)g_2\big)}  \\ +\ang{ A(\lambda)g_2\mid \frac{k^2}{r^2}\big(f(g_1+g_2) - f(g_1) - g_2\big)}\Big| 
        \leq \frac{c_0}{\lambda} \|g_2\|_H^2.
      \end{multline}
    Here the expression inside the absolute value on the left-hand side is defined by the integration-by-parts formula:
    \begin{align}
&\left\langle A(\lambda)g_1 \,\middle|\,
\frac{k^2}{r^2}\big(f(g_1+g_2)-f(g_1)-f'(g_1)g_2\big)
\right\rangle \notag\\
&\quad+
\left\langle A(\lambda)g_2 \,\middle|\,
\frac{k^2}{r^2}\big(f(g_1+g_2)-f(g_1)-g_2\big)
\right\rangle \notag\\
\quad=&
-k^2\int_0^\infty
\left(\frac1r q'\left(\frac r\lambda\right)\right)'
\left[
F(g_1+g_2)-F(g_1)-f(g_1)g_2-\frac12 g_2^2
\right]\,dr .
\end{align}
    where $F(u)=\frac{1}{2}\sin^2u$,

    \item For all $g \in H$ we have  
\EQ{
        \label{eq:pohozaev}
        \ang{\U A(\lambda)g \mid \LL_0 g} \ge  - \frac{c_0}{\lambda}\|g\|_{H}^2 + \frac{1}{\lambda}\int_{R^{-1} \lam}^{R\lambda}  \Big((\partial_r g)^2 + \frac{k^2}{r^2}g^2\Big) \udr, 
        }
        here the left-hand side is defined by the integration-by-parts formula:
        \begin{align}
\la A(\lam)g \mid L_0 g \ra &= \frac{1}{\lam} \int_0^\infty q''\left(\frac{r}{\lam}\right) (\p_r g)^2 r \, \ud r \\
&\quad + k^2 \int_0^\infty q'\left(\frac{r}{\lam}\right) \frac{g^2}{r^2} \, \ud r \\
&\quad - \frac{1}{4} \int_0^\infty \left[ \p_r^2 + \frac{1}{r}\partial_r \right] \left[ \frac{1}{\lam}q''\left(\frac{r}{\lam}\right) + \frac{1}{r}q'\left(\frac{r}{\lam}\right) \right] g^2 r \, \ud r,
        \end{align}

        \item For $\lam, \mu >0$ with either $\lam/\mu \ll 1$ or $\mu/\lam \ll 1$, 
\begin{align} 
      \label{eq:L0-A0}
      \|\ULam \Lambda Q_\uln{\lambda} - \U A(\lambda)\Lambda Q_{\lambda}\|_{L^2} &\leq c_0, \\
      \label{eq:L-A}
      \|\big(\frac{1}{\lam} \Lambda  - A(\lambda)\big) Q_\lambda\|_{L^\infty} &\leq \frac{c_0}{\lambda},  \\
    \| A(\lam) Q_\mu \|_{L^\infty} + \| \U A(\lam) Q_\mu \|_{L^\infty}  &\lesssim  \frac{1}{\lam}  \min \{ (\lam/ \mu)^k,  (\mu/ \lam)^k \}  \label{eq:A-mismatch} \\
 \| A(\lam) Q_\mu \|_{L^2} + \| \U A(\lam) Q_\mu \|_{L^2} &\lesssim   \min \{ (\lam/ \mu)^k,  (\mu/ \lam)^k \}  \label{eq:A-mismatch-2}
     \end{align} 
     
     \item Lastly, the following localized coercivity estimate holds. Fix any smooth function $\calZ \in L^2 \cap X$ such that $\ang{\calZ \mid \Lam Q} >0$. For any $g \in H, \lam>0$ with $\ang{g \mid \calZ_{\U \lam}} = 0$, 
     \EQ{ \label{eq:coercive}
    \frac{1}{\lambda}\int_{R^{-1} \lam}^{R\lambda}  (\partial_r g)^2 + \frac{k^2}{r^2}g^2 \udr &+   \frac{1}{\lam} \int_{0}^\infty  \big(\frac{1}{2}q''\big(\frac{r}{\lambda}\big) + \frac{\lam}{2r}q'\big(\frac{r}{\lambda}\big)\big) \frac{k^2}{r^2} ( f'(Q_{\lam}) - 1)g^2  \, r \, \ud r  \\
    \ge - \frac{c_0}{\lam} \| g \|_{H}^2 .
     }
  \end{itemize}
\end{lem}

\begin{proof}
    The first property and the last two properties were proved in  Lemma 5.5 in \cite{J-NLW-two-bubble}. As for the second and the third properties, Lemma 5.5 in \cite{J-NLW-two-bubble} presents proofs for functions   $g_1,g_2,g$ in the space   
    \begin{align}
        X:=\{g\in H:\frac{g}{r},\p_r g\in H\}.
    \end{align}
    One can easily extend the results to the case when $X$ is replaced by $H$ using density arguments. Notice that the two pairings in the second property are not used separately for general $g_1,g_2\in H$; only their sum is defined through the displayed integration-by-parts formula.
\end{proof}

For any finite energy blow-up solution to ~\eqref{WM}, we consider the new solution $\vec u$ given by Lemma~\ref{lem:small-radiation-valid}.

\begin{lem}\label{lem:mod}
    For any $c_0>0$, there exists $\de>0$ such that for $t\in(T_+-\de,T_+)$ we have the following decomposition
    \begin{align}
        \vec u(t)=\sum_{j=1}^N\iota_j\vec Q_{\lam_j(t)}+\vec u^*(t)+\vec g(t),
    \end{align}
    where $\vec u^*(t)$ is the solution to ~\eqref{WM} with initial data $\vec u_0^*(r)$ at time $T_+$. 
    
    The decomposition has the following properties:
    \begin{align}\label{eq:mod-parameter-same}
        \left|\frac{\lam_j(t)}{\mu_j(t)}-1\right|\leq \frac{1}{100},
    \end{align}
    \begin{align}\label{eq:ortho-1}
        \lim_{s\to T_+}\left(\sum_{j=1}^{N-1}\frac{\lam_j(s)}{\lam_{j+1}(s)}+\frac{\lam_N(s)}{T_+-s}+\|\vec g(s)\|_{\E}\right)=0,
    \end{align}
    \begin{align}\label{eq:lam'}
        |\lam_j'(t)|\lesssim\|\dot g\|_{L^2},\ \forall\ 1\leq j\leq N,
    \end{align}
    \begin{align}\label{eq:ortho}
         \La\cZ_{\U{\lam_j(t)}},g(t)\Ra=0,\ \forall\ 1\leq j\leq N,
    \end{align}
    here $\cZ(r)$ denotes the truncated version of the zero energy eigenfunction $\Lam Q$ for $-\De+\frac{f'(Q)}{r^2}$:
    \begin{align}
        \cZ(r) :=
    \begin{cases}
    \Lam Q(r), & \text{if } k \ge 3, \\
    \chi(r)\Lam Q(r),
    & \text{if } k=2.
    \end{cases}
    \end{align}
   
    Moreover, if we introduce the corrections $\xi_j(t)$ for $\lam_j(t)$ and $\be_j(t)$ for $\lam_j'(t)$ defined as follows:
    
\begin{align} \label{eq:xi-def}
        \xi_j(t) :=
    \begin{cases}
    \lam_j(t), & \text{if } k \ge 3, \\
    \lam_j(t)
    -\dfrac{\iota_j}{\|\Lam Q\|_{L^2}^2}
    \left\langle
    \chi_{L \lam_j(t)}\Lam Q_{\U{\lam_j(t)}}
    \,\middle|\, g(t)
    \right\rangle,
    & \text{if } k=2,
    \end{cases}
\end{align}
    \begin{align} \label{eq:beta-def}
    \be_j(t)=-\frac{\iota_j}{\|\Lam Q\|_{L^2}^2}\La \Lam Q_{\U{\lam_j(t)}},\dot g\Ra-\frac{1}{\|\Lam Q\|_{L^2}^2}\La \uln A(\lam_j(t))g(t),\dot g(t)\Ra,
\end{align}
    then we have
    \begin{align}\label{eq:correction-lam}
        \left|\frac{\xi_j(t)}{\lam_j(t)}-1\right|\leq c_0,\ \forall \ 1\leq j\leq N-1,
    \end{align}
    \begin{align}\label{eq:correction-lam''}
         |\xi_j'(t)-\be_j(t)|\lesssim c_0\bfd(t), \ \forall\ 1\leq j\leq N-1,
    \end{align}
    \begin{align}\label{eq:beta-upper}
        |\be_j(t)|\lesssim \bfd(t), \ \forall\ 1\leq j\leq N-1,
    \end{align}
    \begin{align}\label{eq:beta'}
        \lambda_j(t)\beta_j'(t)\geq \left(-\iota_j\iota_{j+1}\om^2-c_0\right)\left(\frac{\lam_{j}(t)}{\lam_{j+1}(t)}\right)^k+\left(\iota_j\iota_{j-1}\om^2-c_0\right)\left(\frac{\lam_{j-1}(t)}{\lam_{j}(t)}\right)^k-c_0\bfd^2(t),  \ \forall\ 1\leq j\leq N-1,
    \end{align}
    with the convention that $\lam_0=0$.

Here the function $\bfd(t)$ is defined as 
\begin{align}
    \bfd(t)=\sum_{j=1}^{N-1}\left(\frac{\lam_j(t)}{\lam_{j+1}(t)}\right)^{\frac{k}{2}}+\lam_N(t)+\|\vec g\|_{\E},
\end{align}
and the constant $\om^2$ is defined by
\begin{align}
    \om^2=\om^2(k)=4k^2\pi^{-1}\sin\frac{\pi}{k}.
\end{align}

\end{lem}

\begin{rem}
    The function $\bfd(t)$ can be understood intuitively in the following way:  the sum reflects interactions between the bubbles, and the second term the interaction of the widest bubble with radiation. The definition of $\bfd(t)$ collects the quantities whose squares enter the energy and modulation estimates at the same level, as shown in the proof.
\end{rem}

\begin{proof}
    The decomposition and ~\eqref{eq:mod-parameter-same}, ~\eqref{eq:ortho-1},  ~\eqref{eq:ortho}  are guaranteed by Theorem~\ref{thm:src-blowup} and the implicit function theorem.

    In the rest of the proof, we often omit the variables $t$ and $r$ when no
confusion is possible. For instance, $Q_{\lam_j}$, $u^*$, $g$, and
$\lam_j$ stand for $Q_{\lam_j(t)}(r)$, $u^*(t,r)$, $g(t,r)$, and
$\lam_j(t)$, respectively. Static objects such as  $Q$, $\Lambda Q$, and
$A(\lambda)$ with a fixed parameter $\lambda>0$ are not affected by this
convention. We will restore the variables whenever ambiguity may arise.
    
    We first consider ~\eqref{eq:lam'}, ~\eqref{eq:correction-lam}, ~\eqref{eq:correction-lam''} and ~\eqref{eq:beta-upper}.
    
    A basic calculation yields
    \begin{align}\label{eq:p_tg}
        \p_t g(t)&= \p_t u(t)-\p_t u^*(t)+\sum_{j=1}^N \iota_j\lam_j'(t)\Lam Q_{\U{\lam_j(t)}}=\dot g+\sum_{j=1}^N \iota_j\lam_j'(t)\Lam Q_{\U{\lam_j(t)}},
    \end{align}
    which implies that, when differentiating the orthogonality conditions ~\eqref{eq:ortho}, no additional term involving $\partial_tu^*$ appears. Therefore, the proofs of ~\eqref{eq:lam'}, ~\eqref{eq:correction-lam}, ~\eqref{eq:correction-lam''}, and~\eqref{eq:beta-upper}  are  almost the same as those in Lemma 4.11 and Lemma 4.16 in \cite{JL-WM} . We omit the details.

    Now we consider  ~\eqref{eq:beta'}. We calculate
    \begin{equation}\label{eq:p_tdotg}
        \begin{aligned}[t]
            \p_t \dot g(t)=&-\frac{k^2}{r^2}f\left(\sum_{l=1}^{N}\iota_l Q_{\lambda_l(t)}+u^*(t)+g(t) \right)+\frac{k^2}{r^2}f\left(\sum_{l=1}^{N}\iota_l Q_{\lambda_l(t)}+u^*(t)\right)\\
            \MoveEqLeft[-3] +\frac{k^2}{r^2}f'\left(\sum_{l=1}^{N}\iota_l Q_{\lambda_l(t)}+u^*(t)\right)g(t)\\
        &-\frac{k^2}{r^2}\left(f'\left(\sum_{l=1}^{N}\iota_l Q_{\lambda_l(t)}+u^*(t)\right)- f'\left(\sum_{l=1}^{N}\iota_l Q_{\lambda_l(t)}\right)\right)g(t)\\
        &-\frac{k^2}{r^2}f'\left(\sum_{l=1}^{N}\iota_l Q_{\lambda_l(t)}\right)g(t)+\De g(t)\\
        &-\frac{k^2}{r^2}f\left(\sum_{j=1}^{N}\iota_j Q_{\lam_j(t)}+u^*(t)\right)+\frac{k^2}{r^2}f\left(\sum_{j=1}^{N}\iota_j Q_{\lam_j(t)}\right)+\frac{k^2}{r^2}f\left(u^*(t)\right)\\
        &-\frac{k^2}{r^2}f\left(\sum_{j=1}^{N}\iota_j Q_{\lam_j(t)}\right)+\frac{k^2}{r^2}\sum_{j=1}^{N}\iota_jf(Q_{\lam_j(t)}).
        \end{aligned}
    \end{equation}

To simplify the notation, we introduce the following quantities.
    \begin{align}
        &\calL_{\calQ}g=-\De g+\frac{k^2}{r^2}f'\left(\sum_{l=1}^{N}\iota_l Q_{\lambda_l(t)}\right)g,  \label{eq:def-LQ-g}\\
        &\phi_{rg}(t)=-\frac{k^2}{r^2}\left(f'\left(\sum_{l=1}^{N}\iota_l Q_{\lambda_l(t)}+u^*(t)\right)- f'\left(\sum_{l=1}^{N}\iota_l Q_{\lambda_l(t)}\right)\right)g(t), \label{eq:def-phi-rg}\\
        &\phi_{rb}(t)=-\frac{k^2}{r^2}f\left(\sum_{l=1}^{N}\iota_l Q_{\lambda_l(t)}+u^*(t)\right)+\frac{k^2}{r^2}f\left(\sum_{l=1}^{N}\iota_l Q_{\lambda_l(t)}\right)+\frac{k^2}{r^2}f\left(u^*(t)\right), \label{eq:def-phi-rb}\\
        &\phi_{bb}(t)=-\frac{k^2}{r^2}f\left(\sum_{l=1}^{N}\iota_l Q_{\lambda_l(t)}\right)+\sum_{l=1}^{N}\iota_l \frac{k^2}{r^2}f(Q_{\lam_l(t)}), \label{eq:def-phi-bb}\\
        &\phi_q(t)=-\frac{k^2}{r^2}f\left(\sum_{l=1}^{N}\iota_l Q_{\lambda_l(t)}+u^*(t)+g(t) \right)+\frac{k^2}{r^2}f\left(\sum_{l=1}^{N}\iota_l Q_{\lambda_l(t)}+u^*(t)\right)\\
        \MoveEqLeft[-3.4]+\frac{k^2}{r^2}f'\left(\sum_{l=1}^{N}\iota_l Q_{\lambda_l(t)}+u^*(t)\right)g(t), \label{eq:def-phi-q}\\
        &\ti \phi_{q}(t)=-\frac{k^2}{r^2}f\left(\sum_{l=1}^{N}\iota_l Q_{\lambda_l(t)}+u^*(t)+g(t) \right)+\frac{k^2}{r^2}f\left(\sum_{l=1}^{N}\iota_l Q_{\lambda_l(t)}+u^*(t)\right)+\frac{k^2}{r^2}g. \label{eq:def-ti-phi-q}
    \end{align}

    The subscript $r$ above stands for ``radiation'', $b$ stands for ``bubble'' and $q$ stands for ``quadratic''. 

    With this notation, we write
    \begin{align}
        \p_t \dot g&= -\calL_{\calQ}g+\phi_q+\phi_{rb}+\phi_{bb}+\phi_{rg}\\
        &=-\cL_0 g+\ti \phi_q+\phi_{rb}+\phi_{bb},
    \end{align}
    where 
    \begin{align}
        \cL_0 g=-\De g+\frac{k^2}{r^2}g.
    \end{align}

Differentiating \eqref{eq:beta-def} and using the equation for $\p_t\dot g$, we obtain
      \begin{align}\label{eq:beta'-1} 
\| \Lam Q \|_{L^2}^2 \beta_j' =&\iota_j  \frac{\lam_j'}{\lam_j} \La \ULam \Lam Q_{\U{\lam_j}} \mid \dot g \Ra  - \iota_j \La \Lam Q_{\U{\lam_j}} \mid \p_t  \dot g \Ra\\
&-\frac{\lam_j'}{\lam_j} \ang{ \lam_j \p_{\lam_j} \uln A( \lam_j) g \mid \dot g} -    \ang{ \uln A( \lam_j) \p_t g \mid \dot g} -     \ang{ \uln A( \lam_j) g \mid \p_t \dot g}\\
=&\iota_j\frac{\lam_j'}{\lam_j} \La \ULam \Lam Q_{\U{\lam_j}} \mid \dot g \Ra -\frac{\lam_j'}{\lam_j} \ang{ \lam_j \p_{\lam_j} \uln A( \lam_j) g \mid \dot g}\\ &-\iota_j\La\Lam Q_{\U{\lam_j}} \mid -\calL_{\calQ}g+\phi_q+\phi_{rb}+\phi_{bb}+\phi_{rg}\Ra\\
        &-\ang{ \uln A( \lam_j) g \mid  -\cL_0 g+\ti\phi_q+\phi_{rb}+\phi_{bb}}\\
        &-\iota_j \lam'_j(t)\la \uln A( \lam_j)\Lam Q_{\U {\lam_j(t)}}  \mid \dot g\ra-\sum_{l\neq j}\iota_l \lam'_l(t)\la \uln A( \lam_j)\Lam Q_{\U {\lam_l(t)}}  \mid \dot g\ra 
\end{align}

We rearrange the terms into

\begin{align}
    \| \Lam Q \|_{L^2}^2 \beta_j' =&  -  \frac{\iota_j }{\lam_j}  \ang{ \Lam Q_{\lam_j} \mid \phi_{bb} } +  \ang{ \uln A( \lam_j) g \mid \LL_0 g } \\
& + \ang{ (A(\lam_j) - \uln A( \lam_j)) g \mid \ti  \phi_q} \\
&+\iota_j \ang{ \Lam Q_{\U{\lam_j}} \mid ( \LL_{\calQ} - \LL_{\lam_j}) g} + \iota_j \frac{\lam_j' }{\lam_j}\ang{ \big( \frac{1}{\lam_j} \ULam - \U{A}( \lam_j) \big) \Lam Q_{\lam_j} \mid \dot g} \\
&   - \ang{ {A}(\lam_j) \left(\sum_{i =1}^N \iota_i Q_{\lam_i} +u^*\right)\mid \phi_q}   - \ang{ A( \lam_j) g \mid  \ti\phi_q} \\
&   + \iota_j \ang{ ({A}( \lam_j) -  \frac{1}{\lam_j}\Lam) Q_{\lam_j} \mid \phi_q}  -    \frac{\lam_j'}{\lam_j} \ang{ \lam_j \p_{\lam_j} \uln A( \lam_j) g \mid \dot g} \\
& +\ang{ {A}(\lam_j) u^*\mid \phi_q} 
\\
& + \sum_{ i \neq j} \iota_i \ang{ {A}(\lam_j) Q_{\lam_i} \mid \phi_q} \\ 
& - \sum_{i \neq j} \iota_i \lam_{i}'  \ang{ \uln A( \lam_j) \Lam Q_{\U{\lam_i}} \mid \dot g}   - \ang{ \uln A( \lam_j) g \mid \phi_{bb} } \\
&-\iota_j\La\Lam Q_{\U{\lam_j}} \mid \phi_{rb}+\phi_{rg}\Ra\\
        &-\ang{ \uln A( \lam_j) g \mid  \phi_{rb}}.
\end{align}

The difference between the equation above and that in \cite{JL-WM} is that we have to consider the terms containing radiation, i.e. the terms containing $\phi_q,\ti\phi_q,\phi_{rb},\phi_{rg}$.

\emph{Estimates of the terms involving  $\phi_q$, $\ti\phi_q$:} The estimates are the same as those for $f_{\bfq}$ and $\ti f_{\bfq}$ in \cite{JL-WM}, and we record them for the reader's convenience.

We first notice that by Taylor's expansion we have the pointwise bound $|\phi_q(t,r)|\lesssim \frac{|g(t,r)|^2}{r^2}$, which implies
\begin{align}
    \label{eq:phi-q-est} 
\| \phi_q  \|_{L^1} \lesssim \| g \|_{H}^2. 
\end{align}

We treat the terms $\ang{ \uln A( \lam_j) g \mid \LL_0 g } $ and $\ang{ (A(\lam_j) - \uln A( \lam_j)) g \mid \ti  \phi_q} $ together. First, ~\eqref{eq:pohozaev} yields 
\begin{align}
    \ang{ \uln A( \lam_j) g \mid \LL_0 g } \ge -\frac{c_0}{\lam_j} \| g \|_{H}^2 + \frac{1}{\lam_j} \int_{R^{-1} \lam_j}^{R \lam_j} \Big( ( \p_r g)^2 + \frac{k^2}{r^2} g^2  \Big) \, r \ud r.
\end{align}

To treat the term $\ang{ (A(\lam_j) - \uln A( \lam_j)) g \mid \ti  \phi_q} $, we start by combining the definition ~\eqref{eq:def-phi-q} with ~\eqref{eq:def-ti-phi-q} to observe the identity, 
\begin{align}\label{eq:ti-phi-exp} 
    \ti  \phi_{q}%
& = - \frac{k^2}{r^2}( f'(Q_{\lam_j}) - 1)g  - \frac{k^2}{r^2} \left( f'\left( \sum_{l=1}^N \iota_lQ_{\lam_l}+u^*\right) - f'(Q_{\lam_j})\right)g + \phi_q
\end{align}

Next, by definition,  
\begin{align}
    (A( \lam_j) - \U A( \lam_j)) g = - \frac{1}{\lam_j}  \big(\frac{1}{2}q''\big(\frac{r}{\lambda_j}\big) + \frac{\lam_j}{2r}q'\big(\frac{r}{\lambda_j}\big)\big)g
\end{align}

The contributions of the last two terms in ~\eqref{eq:ti-phi-exp} yield acceptable errors. Indeed, 
\begin{align}
    \big| \La (A( \lam_j) - \U A( \lam_j)) g & \mid  \frac{k^2}{r^2} \left( f'\left( \sum_{l=1}^N \iota_lQ_{\lam_l}+u^*\right) - f'(Q_{\lam_j})\right)g \Ra \Big| \\
& \lesssim \frac{1}{\lam_j} \int_{\ti  R^{-1} \lam_j}^{\ti R \lam_j} g^2  \abs{ f'\left( \sum_{l=1}^N \iota_lQ_{\lam_l}+u^*\right) - f'(Q_{\lam_j}))} \, \frac{\ud r}{r}  \le c_0 \frac{ \| g\|_H^2}{\lam_j}
\end{align}

with $\ti R$ as in Lemma~\ref{lem:q}, and by ~\eqref{eq:phi-q-est} and the definition of $q$ from Lemma~\ref{lem:q}, 
\begin{align}
    \abs{\ang{(A( \lam_j) - \U A( \lam_j)) g  \mid  \phi_q }} \lesssim \frac{1}{\lam_j} \|g \|_{L^\infty} \| g \|_H^2  \le c_0 \frac{ \| g\|_H^2}{\lam_j}
\end{align}
Putting this together, we obtain
\begin{align}
    \Big|  \La (A(\lam_j) - \uln A( \lam_j)) g \mid \ti  \phi_{q}\Ra &-  \frac{1}{\lam_j} \int_{0}^\infty  \big(\frac{1}{2}q''\big(\frac{r}{\lambda_j}\big) + \frac{\lam_j}{2r}q'\big(\frac{r}{\lambda_j}\big)\big) \frac{k^2}{r^2} ( f'(Q_{\lam_j}) - 1)g^2  \, r \, \ud r \Big| \\
&\lesssim c_0 \frac{ \| g\|_H^2}{\lam_j}. 
\end{align}

We show that the remaining terms containing $\phi_q$ and $\ti \phi_q$ contribute acceptable errors.

We now estimate the term $- \ang{ {A}(\lam_j) \left(\sum_{i =1}^N \iota_i Q_{\lam_i} +u^*\right)\mid \phi_q}   - \ang{ A( \lam_j) g \mid  \ti\phi_q}$. Applying ~\eqref{eq:A-by-parts}  with $g_1=\sum_{i =1}^N \iota_i Q_{\lam_i} +u^*$ and $g_2=g$, we obtain
\begin{align}
    &\left|\ang{ {A}(\lam_j) \left(\sum_{i =1}^N \iota_i Q_{\lam_i} +u^*\right)\mid \phi_q}   + \ang{ A( \lam_j) g \mid  \ti\phi_q} \right|
    \leq  2c\int_{\ti R^{-1} \lam_j}^{\ti R \lam_j} \frac{1}{\lam_j r} |g(t,r)|^2 \ud r\leq \frac{c_0}{\lam_j}\|g\|_H^2
\end{align}
if we choose $c$ small enough.

By~\eqref{eq:L-A} and ~\eqref{eq:phi-q-est}, we see that
\begin{align}
    \abs{ \ang{ ({A}( \lam_j) -  \frac{1}{\lam_j}\Lam) Q_{\lam_j} \mid \phi_q}     } \lesssim \frac{c_0}{ \lam_j}  \| g \|_{H}^2 . 
\end{align}

Next, using the boundedness of $A(\lam)$ and $|\phi_q(t,r)|\lesssim \frac{|g(t,r)|^2}{r^2}$, we obtain
\begin{align}
    \abs{\ang{ {A}(\lam_j) u^*\mid \phi_q} }\leq & \|A(\lam_j)u^*\|_{L^2}\|\phi_q\|_{L^2(\ti R^{-1}\lam_j,\ti R\lam_j)}\\
    \lesssim &  \|u^*\|_{H} \left(\int_{\ti R^{-1}\lam_j}^{\ti R\lam_j} |\phi_q|^2\udr\right)^{\frac{1}{2}}\\
    \lesssim & \|u^*\|_{H} \left(\int_{\ti R^{-1}\lam_j}^{\ti R\lam_j} \frac{|g(t,r)|^4}{r^4}\udr\right)^{\frac{1}{2}}\\
    \lesssim & \|u^*\|_{H} \left(\int_{\ti R^{-1}\lam_j}^{\ti R\lam_j} \frac{|g(t,r)|^2}{r^4}\udr\right)^{\frac{1}{2}}\| g\|_{H}\\
    \lesssim & \|u^*\|_{H} \left(\int_{\ti R^{-1}\lam_j}^{\ti R\lam_j} \frac{|g(t,r)|^2}{r^2}\udr\right)^{\frac{1}{2}}\frac{\ti R}{\lam_j}\| g\|_{H}\\
    \lesssim & \frac{c_0}{ \lam_j}  \| g \|_{H}^2 . 
\end{align}

Finally, combining~\eqref{eq:A-mismatch} with~\eqref{eq:phi-q-est}, we obtain
\begin{align}
    \Big| \sum_{ i \neq j} \iota_i \ang{ {A}(\lam_j) Q_{\lam_i} \mid  \phi_q}  \Big| \lesssim \frac{c_0}{ \lam_j}  \| g \|_{H}^2 .
\end{align}

\emph{Estimates of the terms involving $\phi_{rb}$:} We first estimate the term $\La\Lam Q_{\U{\lam_j}} \mid \phi_{rb}\Ra$. 

We begin with the pointwise estimate
\begin{align}
    |\phi_{rb}|=&\left|\frac{k^2}{r^2}f\left(\sum_{l=1}^{N}\iota_l Q_{\lam_l(t)}+u^*(t)\right)-\frac{k^2}{r^2}f\left(\sum_{l=1}^{N}\iota_l Q_{\lam_l(t)}\right)-\frac{k^2}{r^2}f\left(u^*(t)\right) \right|\\
    \lesssim &\frac{1}{r^2}\left|\sin(2u^*(t))\left(\cos \left( 2\sum_{l=1}^{N}\iota_l Q_{\lam_l(t)}\right)-1\right)\right|+\frac{1}{r^2}\left|\sin\left(2\sum_{l=1}^{N}\iota_l Q_{\lam_l(t)}\right)(\cos (2u^*(t))-1) \right|\\
    \lesssim & \frac{1}{r^2}\sum_{l=1}^NF_0(r)F_{l}^2(t,r)+\frac{1}{r^2}\sum_{l=1}^NF_0^2(r)F_{l}(t,r).
\end{align}

Since $|\lam_j(t)\Lam Q_{\U {\lam_j(t)}}(r)|\lesssim F_j(t,r)$,  using ~\eqref{eq:appendix-2}, ~\eqref{eq:appendix-3}, ~\eqref{eq:appendix-4} and ~\eqref{eq:appendix-8} we have for $k=2$,
\begin{align}
    \left|\lam_j\La\Lam Q_{\U{\lam_j}} \mid \phi_{rb}\Ra\right|\lesssim & \sum_{l=1}^N\int_0^{\infty}F_j(t,r)F_0(r)F_{l}^2(t,r)\frac{\ud r}{r}+\sum_{l=1}^N\int_0^{\infty}F_j(t,r)F_0^2(r)F_{l}(t,r)\frac{\ud r}{r}\\
    \lesssim & \int_0^{\infty} F_0(r)F_j^3(t,r)\frac{\ud r}{r}+\sum_{l\neq j}\int_0^{\infty}F_j(t,r)F_0(r)F_{l}(t,r)\frac{\ud r}{r}\\
     & + \int_0^{\infty} F_0^2(r)F_j^2(t,r)\frac{\ud r}{r}+\sum_{l\neq j}\int_0^{\infty}F_j(t,r)F^2_0(r)F_{l}(t,r)\frac{\ud r}{r}\\
     \lesssim &\ \lam_j^2(1+|\log \lam_j|)^{\frac{1}{2}}+\sum_{l<j}\lam_l^2(1+|\log \lam_j|)^{\frac{1}{2}}\\
     &+ \sum_{l>j}\lam_j^2(1+|\log \lam_l|)^{\frac{1}{2}}+\lam_j^4(1+|\log \lam_j|)^2\\
     &+ \sum_{l<j} \lam_l^2\lam_j^2(1+|\log \lam_j|)^2+ \sum_{l>j} \lam_l^2\lam_j^2(1+|\log \lam_l|)^2.
\end{align}

We use $j<N$ to obtain
\begin{align}
    \lam_j^2(1+|\log \lam_j|)^{\frac{1}{2}}=\lam_N^{\frac{3}{2}}\frac{\lam_j^{\frac{3}{2}}}{\lam_N^{\frac{3}{2}}}\lam_j^{\frac{1}{2}}(1+|\log\lam_j|)^{\frac{1}{2}}\lesssim \lam_N^3+\left(\frac{\lam_j}{\lam_N}\right)^3\lam_j(1+|\log\lam_j|)=o_{t\to T_+}(1)\bfd^2(t).
\end{align}
The other terms can be treated similarly, so we get
\begin{align}
    \left|\lam_j\La\Lam Q_{\U{\lam_j}} \mid \phi_{rb}\Ra\right|\lesssim o_{t\to T_+}(1)\bfd^2(t).
\end{align}

The case $k\geq3$ is similar and we omit it.

The term $\La\uln A( \lam_j) g\mid \phi_{rb}\Ra$ can be treated using the localization property of $\uln A$. More precisely, for $k=2$ we have 
\begin{align}
    \left|\lam_j\La\uln A( \lam_j) g\mid \phi_{rb}\Ra \right|^2\lesssim & \lam_j^2\|\uln A( \lam_j) g\|_{L^2}^2\int_{\ti R^{-1}\lam_j}^{\ti R\lam_j} |\phi_{rb}|^2 r\ud r\\
    \lesssim & \lam_j^2\|\vec g\|_{\E}^2\int_{\ti R^{-1}\lam_j}^{\ti R\lam_j}  \frac{|F_0(\lam_j)|^2}{r^4} r\ud r\\
    \lesssim & \|\vec g\|_{\E}^2|F_0(\lam_j)|^2\\
    \lesssim & \|\vec g\|_{\E}^2 \lam_j^4(1+|\log\lam_j|)\\
    =& o_{t\to T_+}(1)\bfd^4(t).
\end{align}

\emph{Estimate of the term involving $\phi_{rg}$:}  It is very similar to that for $\phi_{rb}$ and we omit it.

\end{proof}

In the rest of this paper, we will fix the solution $\vec u$ given in Lemma~\ref{lem:small-radiation-valid} and only consider the decomposition given in Lemma~\ref{lem:mod}.

\section{Energy estimate}\label{Sec:4}

Throughout this section, we use the same convention as in the proof of
Lemma~\ref{lem:mod}: we may suppress the variables $t$ and $r$ when doing so
improves readability. All time-dependent quantities are understood to be
evaluated at the same time $t$. We keep the variables explicit whenever they
are useful for clarity, in particular in statements of estimates and in
formulas involving time differentiation.

Define
\begin{align}
    \A=\{1\leq j\leq N-1:\iota_j\neq \iota_{j+1}\}.
\end{align}

In this section, we derive the energy expansion starting from~\eqref{eq:energy-decoupling}:
\begin{align}
    &NE(Q)+E(u^*)=E(u)\\
    =&E(\sum_{j=1}^N\iota_j\vec Q_{\lam_j(t)}+\vec u^*(t))+\La \vD E(\sum_{j=1}^N\iota_j\vec Q_{\lam_j(t)}+\vec u^*(t)),\vec g(t)\Ra\\
    &+\frac{1}{2}\La\vD^2E(\sum_{j=1}^N\iota_j\vec Q_{\lam_j(t)}+\vec u^*(t) )\vec g(t),\vec g(t)\Ra+O(\|\vec g(t)\|_{\E}^3).
\end{align}

\begin{lem}\label{lem:energy-estimate-main}
    We have
    \begin{align}\label{eq:energy-estimate-main}
    \|\vec g\|_{\E}^2\lesssim&  \sum_{j\in \A,1\leq j\leq N-1}\left(\frac{\lam_{j}}{\lam_{j+1}}\right)^k-\sum_{j\notin \A,1\leq j\leq N-1}\left(\frac{\lam_{j}}{\lam_{j+1}}\right)^k\\
    &+ \lam_N^2+\left|\La \vD E(\vec u^*),\vec g\Ra\right|,
\end{align}

with
\begin{align}\label{eq:F'-control-main}
        \left|\frac{\ud }{\ud t}\La\vD E(\vec u^*(t)),\vec g(t)\Ra\right|\lesssim  \bfd^2(t)+\lam_N^2(t)|\log \lam_N(t)|.
    \end{align}
\end{lem}

\begin{rem}
    From the proof we can see that when $k\geq 3$,  we actually have
    \begin{align}
        \left|\frac{\ud }{\ud t}\La\vD E(\vec u^*(t)),\vec g(t)\Ra\right|\lesssim  \bfd^2(t),
    \end{align}
    which makes the analysis in Section~\ref{sec:proof} much simpler. However we will treat the case $k\geq 2$ together.
\end{rem}

This lemma is an immediate result of the remaining parts of this section.

\subsection{Zeroth order term}
\begin{lem}\label{lem:zeroth-order-term}
Fix $k\ge1,  M \in \N$. 
For any $\te>0$, there exists $\eta>0$ with the following property. Consider the subset of $M$-bubble configurations  $\sum_{j=1}^M\iota_j\vec Q_{\lam_j}$ such that 
\begin{align}
\sum_{j =1}^{M-1} \Big( \frac{ \lam_{j}}{\lam_{j+1}} \Big)^k \le \eta.
\end{align}
Then, 
\begin{align}\label{eq:main-term-bubble-interaction}
  \Big|  E\left( \sum_{j=1}^M\iota_j\vec Q_{\lam_j}\right)  - M E( \vec Q) -  16 k \pi \sum_{j =1}^{M-1} \iota_j \iota_{j+1}  \Big( \frac{ \lam_{j}}{\lam_{j+1}} \Big)^k  \Big| \le \te \sum_{j =1}^{M-1} \Big( \frac{ \lam_{j}}{\lam_{j+1}} \Big)^k.
\end{align}

In our modulation setting we also have
\begin{align}\label{eq:bubble-interaction-energy-difference}
    \left|E\left(\sum_{j=1}^N\iota_j\vec Q_{\lam_j(t)}\right)+E(\vec u^*(t))-E\left(\sum_{j=1}^N\iota_j\vec Q_{\lam_j(t)}+\vec u^*(t)\right)\right|\lesssim\lam_N^2(t)+\bfd^2(t) \cdot o_{t\to T_+}(1).
\end{align}
\end{lem}

\begin{proof}
    For the proof of ~\eqref{eq:main-term-bubble-interaction}, see Lemma 2.22 in \cite{JL-WM}. We only give a proof of ~\eqref{eq:bubble-interaction-energy-difference}. 

    A direct computation yields
\begin{align}
    &E(\sum_{j=1}^N\iota_j \vec Q_{\lam_j(t)}) + E(\vec u^*(t)) - E(\sum_{j=1}^N\iota_j  \vec Q_{\lam_j(t)} + \vec u^*(t)) \\=& \pi k^2 \int_0^\infty \left[ \sin^2\left(\sum_{j=1}^N\iota_j  Q_{\lam_j(t)}\right) + \sin^2 u^*(t) - \sin^2\left(\sum_{j=1}^N\iota_j  Q_{\lam_j(t)} + u^*(t)\right)\right]\frac{\ud r}{r}\\
    &+\pi k^2\int_0^{\infty}\left[u^*(t) \sum_{j=1}^N \iota_j \sin(2Q_{\lam_j(t)}) \right] \frac{\ud r}{r}\\
    =& \pi k^2 \int_0^\infty \left[ \sin\left(2\sum_{j=1}^N\iota_j  Q_{\lam_j(t)}\right) \left( u^*(t,r) - \frac{\sin(2u^*(t,r))}{2} \right) + 2 \sin^2 \left(\sum_{j=1}^N\iota_j  Q_{\lam_j(t)}\right) \sin^2 u^*(t,r)  \right] \frac{\ud r}{r}     \\
    &+\pi k^2 \int_0^\infty u^*(t,r) \left[ \sum_{j=1}^N \iota_j \sin(2Q_{\lam_j(t)}) - \sin\left(2\sum_{j=1}^N\iota_j  Q_{\lam_j(t)}\right) \right] \frac{\ud r}{r}.
\end{align}

Since $|u^*|\lesssim\|\vec u^*\|_{\E}\ll1$,  ~\eqref{eq:appendix-1} yields
\begin{align}
    &\left|\int_0^\infty \left[ \sin\left(2\sum_{j=1}^N\iota_j  Q_{\lam_j(t)}\right) \left( u^*(t,r) - \frac{\sin(2u^*(t,r))}{2} \right)  \right] \frac{\ud r}{r}   \right|\\
    \lesssim&  \sum_{j=1}^N \int_0^{\infty}  \left|\sin 2Q_{\lam_j(t)}\right| \left|u^*(t,r)\right|^2\frac{\ud r}{r}\\
    \lesssim&\sum_{j=1}^N\int_0^{\infty}F_j(t,r)F_0^2(r)\frac{\ud r}{r} \\
    \lesssim&\sum_{j=1}^N\lam_j^2(t)\lesssim\lam_N^2(t).
\end{align}

Using ~\eqref{eq:appendix-2} and ~\eqref{eq:appendix-3},  we obtain
\begin{align}
    \left|\int_0^\infty  \sin^2\left(\sum_{j=1}^N\iota_j  Q_{\lam_j(t)}\right) \sin^2 u^*(t,r)   \frac{dr}{r}   \right|&\lesssim \sum_{i,j=1}^N\int_0^{\infty}F_i(t,r)F_j(t,r)F_0^2(r)\frac{\ud r}{r}\\
    &\lesssim\sum_{j=1}^N\lam_j^2(t)\lesssim\lam_N^2(t).
\end{align}

Now we treat the last term. 

For $k\geq 3$ we use ~\eqref{eq:sin-decouple} and ~\eqref{eq:appendix-4} to obtain
\begin{align}
    &\left|\int_0^\infty u^* (t,r)\left[ \sum_{j=1}^N \iota_j \sin(2Q_{\lambda_j(t)}) - \sin\left(2\sum_{j=1}^N\iota_j  Q_{\lam_j(t)}\right) \right] \frac{dr}{r}\right|\\
    \lesssim& \sum_{1\leq i<j\leq N}\int_0^{\infty}F_0(r)F_i(t,r)F_j(t,r)\frac{\ud r}{r}\\
    \lesssim& \sum_{1\leq i<j\leq N} \lam_j^2(t)\left(\frac{\lam_i(t)}{\lam_j(t)}\right)^k\\
    \lesssim& \lam_N^2(t).
\end{align}

For $k=2$ we have
\begin{align}
    &\left|\int_0^\infty u^* (t,r)\left[ \sum_{j=1}^N \iota_j \sin(2Q_{\lambda_j(t)}) - \sin\left(2\sum_{j=1}^N\iota_j  Q_{\lam_j(t)}\right) \right] \frac{\ud r}{r}\right|\\
    \lesssim& \sum_{1\leq i<j\leq N}\int_0^{\infty}F_0(r)F_i(t,r)F_j(t,r)\frac{\ud r}{r}\\
    \lesssim& \sum_{1\leq i<j\leq N}  \lam_i^2(t)(1+|\log \lam_j(t)|)^{\frac{1}{2}}\\
    =& \sum_{1\leq i<j\leq N} \left(\frac{\lam_i(t)}{\lam_j(t)}\right)^2 \lam_j^2(t)(1+|\log \lam_j(t)|)^{\frac{1}{2}}\\
    =& \ \bfd^2(t) \cdot o_{t\to T_+}(1).
\end{align}

\end{proof}

\subsection{First order term}
\begin{lem}
    There holds
    \begin{align}
        &\left|\La \vD E(\sum_{j=1}^N\iota_j\vec Q_{\lam_j(t)}+\vec u^*(t) )-\vD E (\vec u^*(t))\mid \vec g(t)  \Ra\right|\\ \lesssim & \|\vec g\|_{\E}\left(\sum_{j=1}^{N-1}\left(\frac{\lam_j(t)}{\lam_{j+1}(t)}\right)^k+\lam_N^2(t)\left(1+|\log \lam_N(t)|\right)\right)
    \end{align}
\end{lem}

\begin{proof}
    We write
    \begin{align}
        &\La \vD E(\sum_{j=1}^N\iota_j\vec Q_{\lam_j(t)}+\vec u^*(t) )-\vD E (\vec u^*(t))\mid \vec g(t)  \Ra\\
        =& \La \vD E(\sum_{j=1}^N\iota_j\vec Q_{\lam_j(t)} ) \mid \vec g(t)\Ra\\
        &+\frac{k^2}{2} \int_0^{\infty} \left[ \sin\left( 2u^*(t,r) + 2\sum_{j=1}^N \iota_j Q_{\lam_j(t)} \right) - \sin\left( 2\sum_{j=1}^N \iota_j Q_{\lam_j(t)} \right) - \sin(2u^*(t,r)) \right] g(t,r) \frac{\ud r}{r}.
    \end{align}

From Lemma 2.22 in \cite{JL-WM} we know that
\begin{align}
    \left|\La \vD E(\sum_{j=1}^N\iota_j\vec Q_{\lam_j(t)} ) \mid \vec g(t)\Ra\right|\lesssim\|\vec g(t)\|_{\E}\left(\sum_{j=1}^{N-1}\left(\frac{\lam_j(t)}{\lam_{j+1}(t)}\right)^k\right).
\end{align}

Rearranging the term in the last line, we obtain
\begin{align}
    &\frac{k^2}{2} \int_0^{\infty} \left[ \sin\left( 2u^*(t,r) + 2\sum_{j=1}^N \iota_j Q_{\lam_j(t)} \right) - \sin\left( 2\sum_{j=1}^N \iota_j Q_{\lam_j(t)} \right) - \sin(2u^*(t,r)) \right] g(t,r) \frac{\ud r}{r}\\
    =&\frac{k^2}{2} \int_0^{\infty} \left[ \sin\left( 2\sum_{j=1}^N \iota_j Q_{\lam_j(t)} \right) (\cos (2u^*(t,r))-1)\right] g(t,r) \frac{\ud r}{r}\\
    &+\frac{k^2}{2} \int_0^{\infty} \left[ \sin\left( 2u^*(t,r) \right) \left(\cos \left( 2\sum_{j=1}^N \iota_j Q_{\lam_j(t)} \right)-1\right)\right] g(t,r) \frac{\ud r}{r}.
\end{align}

Therefore it follows from ~\eqref{eq:appendix-3} that
\begin{align}
    &\left|\frac{k^2}{2} \int_0^{\infty} \left[ \sin\left( 2u^*(t,r) + 2\sum_{j=1}^N \iota_j Q_{\lam_j(t)} \right) - \sin\left( 2\sum_{j=1}^N \iota_j Q_{\lam_j(t)} \right) - \sin(2u^*(t,r)) \right] g(t,r) \frac{\ud r}{r}\right| \\
    \lesssim & \int_0^{\infty} \left[ \sum_{j=1}^NF_j(t,r) F_0^2(r)\right] |g(t,r)| \frac{\ud r}{r}+\int_0^{\infty} \left[ F_0(r) \sum_{j=1}^NF_j^2(t,r)\right] |g(t,r)| \frac{\ud r}{r}\\
    \lesssim & \sum_{j=1}^N\left(\int_0^{\infty}F_j^2(t,r)F_0^4(r)\frac{\ud r}{r}\right)^{\frac{1}{2}}\left( \int_0^{\infty}\frac{|g(t,r)|^2}{r^2}r\ud r\right)^{\frac{1}{2}}\\
    &+\sum_{j=1}^N \left(\int_0^{\infty}F_j^4(t,r)F_0^2(r)\frac{\ud r}{r}\right)^{\frac{1}{2}}\left(\int_0^{\infty}\frac{|g(t,r)|^2}{r^2}r\ud r\right)^{\frac{1}{2}}\\
    \lesssim & \sum_{j=1}^N \|\vec g(t)\|_{\E}\lam_j^2(t)(1+|\log \lam_j(t)|)\lesssim \|\vec g(t)\|_{\E} \lam_N^2(t)(1+|\log \lam_N(t)|).
\end{align}

\end{proof}

\begin{lem}\label{lem:F'-estimate}
    There holds
    \begin{align}
        \left|\frac{\ud }{\ud t}\La\vD E(\vec u^*(t)),\vec g(t)\Ra\right|\lesssim  \bfd^2(t)+\lam_N^2(t)(1+|\log \lam_N(t)|).
    \end{align}
\end{lem}

\begin{proof}
    We compute its time derivative.
\begin{align}
    &\frac{\ud }{\ud t}\La\vD E(\vec u^*(t)),\vec g(t)\Ra\\
    =&\La\vD^2 E(u^*(t))J\circ \vD E(\vec u^*(t)),\vec g(t)\Ra\\
    &+\La \vD E(\vec u^*(t)), J\circ \vD E(\vec u(t))-J\circ \vD E(\vec u^*(t))+\sum_{j=1}^N\iota_j\lam_j'(t)\Lam\vec Q_{\U{\lam_j(t)}}  \Ra\\
    =&\La \vD E(\vec u^*(t)),J\circ\left(\vD E(\vec u(t))-\vD E(\vec u^*(t))-\vD^2 E(u^*(t))\vec g(t)\right)\Ra\\
    &+\La \vD E(\vec u^*(t)), \sum_{j=1}^N\iota_j\lam_j'(t)\Lam\vec Q_{\U{\lam_j(t)}}  \Ra\\
    =&- \int_0^\infty (\p_t u^*(t,r)) \frac{k^2}{2r^2} \left[ \sin(2u(t,r)) - \sin(2u^*(t,r)) - 2 g(t,r) \cos(2u^*(t,r)) - \sum_{j=1}^N \iota_j \sin(2Q_{\lam_j(t)}) \right] r \ud r \\
    &- \sum_{j=1}^N \iota_j \lam_j'(t) \int_0^\infty (\p_{tt} u^*(t,r)) \left[ \frac{k}{\lam_j(t)} \sin Q_{\lam_j(t)}(r) \right] r \ud r\\
    =&:\mathrm{I}+\mathrm{II}.
\end{align}

We first estimate the term $\mathrm{I}$. Recall that $\rho>0$ provided by Lemma ~\ref{lem:small-radiation-valid} is independent of time, and we allow implicit constants depending on $\rho$.

We decompose the nonlinear term into four parts:
\begin{align}
    &\sin(2u) - \sin(2u^*) - 2 g \cos(2u^*) - \sum_{j=1}^N \iota_j \sin(2Q_{\lam_j})\\
    =&\sin\left(2\sum_{j=1}^N\iota_jQ_{\lam_j}+2u^*+2g\right) - \sin(2u^*) - 2 g \cos(2u^*) \\
    &- \sum_{j=1}^N \iota_j \sin(2Q_{\lam_j})\\
    =&\left(\sin\left(2\sum_{j=1}^N\iota_jQ_{\lam_j}+2u^* \right)- \sin(2u^*)- \sum_{j=1}^N \iota_j \sin(2Q_{\lam_j})\right)\\
    &+\sin\left(2\sum_{j=1}^N\iota_jQ_{\lam_j}+2u^*\right)(\cos 2g-1)\\
    &+\cos(2u^*)\left(\sin(2g)-2g\right)\\
    &+\left(\cos\left(2\sum_{j=1}^N\iota_jQ_{\lam_j}+2u^*\right)-\cos (2u^*)\right)\sin(2g).
\end{align}

We first estimate the last three terms. For the second term, using ~\eqref{eq:radiation-small-lem}, ~\eqref{eq:radiation-decay-ut-outside}, and ~\eqref{eq:radiation-decay-ut-local}, we have

\begin{align}
    &\left|\int_0^{\infty} (\p_t u^*) \frac{k^2}{2r^2} \left[ \sin\left(2\sum_{j=1}^N\iota_jQ_{\lam_j}+2u^*\right)(\cos 2g-1) \right] r \ud r\right|\\
    =&\left|\int_0^{\infty} (\p_t u^*) \frac{k^2}{r^2} \left[ \sin\left(2\sum_{j=1}^N\iota_jQ_{\lam_j}+2u^*\right)(\sin^2 g) \right] r \ud r\right|\\
    \lesssim& \rho\int_0^{\rho}\frac{g^2}{r^2}r\ud r+\left(\int_{\rho}^{8\rho}|\p_t u^*|^2\udr\right)^{\frac{1}{2}}\left(\int_{\rho}^{8\rho}\frac{g^4}{r^4}\udr\right)^{\frac{1}{2}}\\
    \lesssim& \|\vec g(t)\|_{\E}^2.
\end{align}

We estimate the third term noting that

\begin{align}
    g^2(R)=-2\int_R^{\infty}\frac{g(r)}{r}\p_rg(r)r\ud r\lesssim \|\vec g\|_{\E}^2\ll 1,
\end{align}

which combined with ~\eqref{eq:radiation-small-lem} implies

\begin{align}
    &\left|\int_0^{\infty} (\p_t u^*) \frac{k^2}{2r^2} \left[ \cos(2u^*)\sin(2g)-\cos(2u^*)2g\right] r \ud r\right|\\
    \lesssim&\int_0^{\infty} \left|\p_t u^*\right| \frac{k^2}{2r^2}  \left|\cos(2u^*)\right||g|^3 r \ud r\\
    \lesssim & \rho\int_0^{\rho}\frac{g^3}{r^2}r\ud r+\left(\int_{\rho}^{8\rho}|\p_t u^*|^2\udr\right)^{\frac{1}{2}}\left(\int_{\rho}^{8\rho}\frac{g^6}{r^4}\udr\right)^{\frac{1}{2}}\\
    \lesssim&\|\vec g(t)\|_{\E}^2
\end{align}

For the last term,  we use ~\eqref{eq:radiation-decay-ut-local},~\eqref{eq:appendix-5} on $0<r<\rho$, ~\eqref{eq:radiation-small-lem}, $F_j(t,r)\lesssim \lam_j^k(t)$ on $\rho\leq r<8\rho$ and ~\eqref{eq:radiation-decay-ut-outside} on $r\geq 8\rho$ to obtain
\begin{align}
    &\left|\int_0^{\infty} (\p_t u^*) \frac{k^2}{2r^2} \left[ \left(\cos\left(2\sum_{j=1}^N\iota_jQ_{\lam_j}+2u^*\right)-\cos (2u^*)\right)\sin(2g) \right] r \ud r\right|\\
    =&\left|\int_0^{\infty} (\p_t u^*) \frac{k^2}{2r^2} \left[ (-2 \sin\left(2u^* + \sum \iota_j Q_{\lam_j}\right) \sin\left(\sum \iota_j Q_{\lam_j}\right))\sin(2g) \right] r \ud r\right|\\
    \lesssim& \sum_{j=1}^N \int_0^\infty |\partial_t u^*| \frac{1}{r^2} F_j(t,r) |g(t,r)| r \ud r\\
    \lesssim& \|\vec g(t)\|_{\E}\sum_{j=1}^N  \int_0^\rho \ti F_0(r) \frac{1}{r^2} F_j(t,r) r \ud r + \sum_{j=1}^N\left(\int_{\rho}^{8\rho}|\p_t u^*|^2\udr\right)^{\frac{1}{2}}\left(\int_{\rho}^{8\rho}\frac{1}{r^4}F_j^2g^2\udr\right)^{\frac{1}{2}}\\
    \lesssim& \|\vec g(t)\|_{\E}\sum_{j=1}^N  \int_0^\infty \ti F_0(r) \frac{1}{r^2} F_j(t,r) r \ud r + \sum_{j=1}^N \lam_j^k(t)\left(\int_{\rho}^{8\rho}\frac{1}{r^4}g^2\udr\right)^{\frac{1}{2}}\\
    \lesssim& \lam_N(t)\|\vec g(t)\|_{\E}.
\end{align}

Now we estimate the first term

\begin{align}
    \int_0^{\infty} (\p_t u^*) \frac{k^2}{2r^2} \left[ \sin\left(2\sum_{j=1}^N\iota_jQ_{\lam_j}+2u^* \right)- \sin(2u^*)- \sum_{j=1}^N \iota_j \sin(2Q_{\lam_j})\right] r \ud r.
\end{align}

We use ~\eqref{eq:sin-decouple-1}, ~\eqref{eq:appendix-6}, ~\eqref{eq:appendix-7} on $0<r<\rho$, ~\eqref{eq:radiation-small-lem}, $F_0(r)\leq 1$, $F_j(t,r)\lesssim\lam_j^k(t)$ on $\rho\leq r<8\rho$, and ~\eqref{eq:radiation-decay-ut-outside} on $r\geq 8\rho$ to obtain
\begin{align}
    &\left|\int_0^{\infty} (\p_t u^*(t,r)) \frac{k^2}{2r^2} \left[ \sin\left(2\sum_{j=1}^N\iota_jQ_{\lam_j(t)}+2u^*(t,r) \right)- \sin(2u^*(t,r))- \sum_{j=1}^N \iota_j \sin(2Q_{\lam_j(t)})\right] r \ud r\right|\\
    \lesssim& \sum_{0 \leq i < j \leq N} \int_0^\rho \frac{\ti F_0(r)}{r} F_i(t,r) F_j(t,r) \ud r+\sum_{0 \leq i < j \leq N} \left(\int_{\rho}^{8\rho}|\p_t u^*|^2\udr\right)^{\frac{1}{2}}\left(\int_{\rho}^{8\rho}\frac{1}{r^4}F_i^2F_j^2\udr\right)^{\frac{1}{2}}\\
    \lesssim&   \sum_{j=1}^N \lam_j^2(t) + \sum_{1 \leq i < j \leq N} \left(\frac{\lam_i(t)}{\lam_j(t)}\right)^k \lam_j(t)+\sum_{1 \leq i < j \leq N}\lam_i^k(t)\lam_j^k(t)+\sum_{j=1}^N\lam_j^k(t)\\
    \lesssim&   \sum_{j=1}^N \lam_j^2(t) + \sum_{1 \leq i < j \leq N} \left(\frac{\lam_i(t)}{\lam_j(t)}\right)^k \lam_j(t).
\end{align}

Next we estimate the term $\mathrm{II}$.

By definition we have 
\begin{align}
    (\p_{rr}+\frac{1}{r}\p_r)  \sin Q_{\lam_j}=\frac{k^2}{r^2}\cos\left(2Q_{\lam_j}\right)\sin Q_{\lam_j}.
\end{align}

Since $\vec u^*$ satisfies ~\eqref{WM}, integrating by parts we get
\begin{align}
    &\int_0^\infty (\p_{tt} u^*) \left[ \frac{k}{\lam_j} \sin Q_{\lam_j}(r) \right] r \ud r\\
    =&\int_0^{\infty}u^*(\p_{rr}+\frac{1}{r}\p_r)\left( \frac{k}{\lam_j} \sin Q_{\lam_j}(r) \right) r\ud r -\int_0^\infty \frac{k^2}{2r^2} \sin(2u^*)\frac{k}{\lam_j} \sin Q_{\lam_j}(r)  r \ud r\\
    =&  \int_0^\infty \frac{k^2}{2r^2} \Big[ 2u^* \cos(2Q_{\lam_j}) - \sin(2u^*)   \Big] \left[ \frac{k}{\lam_j} \sin Q_{\lam_j}(r) \right] r \ud r\\
   =&\int_0^\infty \frac{k^2}{2r^2} \Big[   2u^* \cos(2Q_{\lam_j})-2u^* \Big] \left[ \frac{k}{\lam_j} \sin Q_{\lam_j}(r) \right] r \ud r
   +\int_0^\infty \frac{k^2}{2r^2} \Big[ 2u^*- \sin(2u^*)  \Big] \left[ \frac{k}{\lam_j} \sin Q_{\lam_j}(r) \right] r \ud r
\end{align}

For the first term,  ~\eqref{eq:appendix-8} gives
\begin{align}
    &\left|\int_0^\infty \frac{k^2}{2r^2} \Big[  2u^* \cos(2Q_{\lam_j})-2u^* \Big] \left[ \frac{k}{\lam_j} \sin Q_{\lam_j}(r) \right] r \ud r\right|\\ 
    \lesssim&\frac{1}{\lam_j}\int_0^{\infty}\frac{1}{r}F_0(r)F_j^3(t,r)\,\ud r
    \lesssim \lam_j(1+|\log \lam_j|)^{\frac{1}{2}}.
\end{align}

For the second term, since $|u^*|\lesssim\|\vec u^*\|_{\E}\ll 1$,   ~\eqref{eq:appendix-9} yields
\begin{align}
    &\left|\int_0^\infty \frac{k^2}{2r^2} \Big[ 2u^*- \sin(2u^*)  \Big] \left[ \frac{k}{\lam_j} \sin Q_{\lam_j}(r) \right] r \ud r\right|\\
    \lesssim& \int_0^\infty \frac{1}{r^2} |u^*|^3 \left| \frac{1}{\lam_j} \sin Q_{\lam_j}(r) \right| r \ud r\\
    \lesssim&\frac{1}{\lam_j}\int_0^{\infty}\frac{1}{r}F_0^3(r)F_j(t,r)\ud r\lesssim \lam_j.
\end{align}

Therefore
\begin{align}
    |\mathrm{II}|\lesssim \sum_{j=1}^N |\lam_j'(t)| \lam_j(t)(1+|\log \lam_j(t)|)^{\frac{1}{2}}.
\end{align}

Combining these estimates, we obtain
\begin{align}
    \left|\frac{\ud }{\ud t}\La\vD E(\vec u^*(t)),\vec g(t)\Ra\right|\lesssim&  \|\vec g(t)\|_{\E}^2+ \sum_{j=1}^N \lambda_j(t) \|\vec g(t)\|_{\E}+\sum_{j=1}^N \lam_j^2(t) \\&+ \sum_{1 \leq i < j \leq N} \left(\frac{\lam_i(t)}{\lam_j(t)}\right)^k \lam_j(t)+\sum_{j=1}^N |\lam_j'(t)| \lam_j(t)(1+|\log \lam_j|)^{\frac{1}{2}},
\end{align}

which, combined with Lemma~\ref{lem:mod}, shows that
\begin{align}
    \left|\frac{\ud }{\ud t}\La\vD E(\vec u^*(t)),\vec g(t)\Ra\right|\lesssim\bfd^2(t)+\lam_N^2(t)(1+|\log \lam_N(t)|).
\end{align}
\end{proof}

\subsection{Second order term}

We first recall the standard coercivity lemma for pure multi-bubble.
    \begin{lem}[\rm{\cite[Lemma 2.19]{JL-WM}}] \label{lem:D2E-coercive-1}
        Fix $k>1$, $M\in \N$. There exist $\eta,c_0>0$ with the following properties.
Consider the subset of $M$-bubble configurations
$\sum_{j=1}^M\iota_jQ_{\lam_j}$ for
$\vec\iota=(\iota_1,\dots,\iota_M)\in\{-1,1\}^M$, $\vec\lam=(\lam_1,\dots,\lam_M)\in(0,\infty)^M$ such that,
\begin{equation}\label{eq:2.19}
\sum_{j=1}^{M-1}
\left(\frac{\lam_j}{\lam_{j+1}}\right)^k
\leq \eta^2 .
\end{equation}
Let $g\in H$ be such that
\[
0=\left\langle \cZ_{\lambda_j}\mid g\right\rangle
\qquad \text{for } j=1,\ldots,M
\]
for some $\vec{\lam}$ as in \eqref{eq:2.19}. Then,
\[
\left\langle
\vD^2 E_p\bigl(\sum_{j=1}^M\iota_jQ_{\lam_j}\bigr)g
\mid g
\right\rangle
\geq c_0\|g\|_{H}^2 .
\]
    \end{lem}

We immediately have

\begin{cor} \label{lem:D2E-coercive}  Fix $k \ge 2$, $M \in \N$. There exist $\eta, c_0,\de>0$ with the following properties. Consider the subset of $M$-bubble configurations $\sum_{j=1}^M\iota_j\vec Q_{\lam_j}$ for $\vec \iota \in \{-1, 1\}^M$, $\vec \lam \in (0, \infty)^M$ and $\vec v\in\E$ such that, 
\begin{align} \label{eq:lam-ratio} 
\sum_{j =1}^{M-1} \Big( \frac{\lam_j}{\lam_{j+1}} \Big)^k \le \eta^2,\\
\|\vec v\|_{\E}<\de.
\end{align}
Let $g \in H$ be such that 
\begin{align}
0 = \ang{\cZ_{\U{\lam_j}} \mid g}  \mfor j = 1, \dots M 
\end{align}
for some $\vec \lam$ as in~\eqref{eq:lam-ratio}. Then, 
\begin{align}
\ang{ \uD^2 E_{\bfp}\left( \sum_{j=1}^M\iota_j Q_{\lam_j}+ v\right) g \mid g} \ge c_0 \| g \|_{H}^2. 
\end{align}
\end{cor} 

\begin{proof}
    We have
    \begin{align}
        &\left|\ang{ \uD^2 E_{\bfp}\left( \sum_{j=1}^M\iota_jQ_{\lam_j}+ v\right) g \mid g} -\ang{ \uD^2 E_{\bfp}\left( \sum_{j=1}^M\iota_j Q_{\lam_j}\right) g \mid g} \right|\\
        =&k^2\left|\int_0^{\infty}\left(\frac{1}{r^2}f'\left(\sum_{j=1}^M\iota_j Q_{\lam_j}+ v \right)-\frac{1}{r^2}f'\left(\sum_{j=1}^M\iota_j Q_{\lam_j} \right)\right)g^2(r)r\,\ud r\right|\\
        =&2k^2\left|\int_0^{\infty}\sin\left(2\sum_{j=1}^M\iota_j Q_{\lam_j}+v\right)\sin\left(v\right)\frac{g^2(r)}{r^2}r\,\ud r\right|\\
        \lesssim&\|v\|_{H}\| g \|_{H}^2,
    \end{align}
    and the rest of the proof follows from Lemma~\ref{lem:D2E-coercive-1}.
\end{proof}

\section{Proof of the main theorem}\label{sec:proof}

In this section we let $F(t):=\La\vD E(\vec u^*(t)),\vec g(t)\Ra$. After a time translation we may assume that Lemma~\ref{lem:mod} and Lemma~\ref{lem:energy-estimate-main} hold for $t\in[0,T_+)$.

Consider 
\begin{align}
    \psi(t)=\sum_{1\leq j\leq N-1, j\in \A} 2^{-j}\xi_j(t)\be_j(t).
\end{align}
We note that a direct analog of this quantity was introduced in \cite{JL-NLW}, motivated by a similar quantity appearing in \cite{DKMActa}.

It follows from ~\eqref{eq:correction-lam}, ~\eqref{eq:correction-lam''}, ~\eqref{eq:beta-upper}, and ~\eqref{eq:beta'} that
\begin{align}\label{eq:psi}
    \left|\frac{\psi(t)}{\lam_N(t)}\right|\lesssim \sum_{1\leq j\leq N-1, j\in \A} \frac{\lam_{j}(t)}{\lam_N(t)}\bfd(t)\lesssim \frac{\lam_{N-1}(t)}{\lam_N(t)}\bfd(t)\lesssim\bfd^{1+\frac{2}{k}}(t)
\end{align}
\begin{align}\label{eq:psi'}
    \psi'(t)\geq& \sum_{j\in \A,1\leq j\leq N-1}2^{-j}\be_j^2(t)+\sum_{j\in \A,1\leq j\leq N-1}2^{-j}\left[\left(-\iota_j\iota_{j+1}\om^2\right)\left(\frac{\lam_{j}(t)}{\lam_{j+1}(t)}\right)^k+\left(\iota_j\iota_{j-1}\om^2\right)\left(\frac{\lam_{j-1}(t)}{\lam_{j}(t)}\right)^k\right]\\
    &-\ti c_0\bfd^2(t)
\end{align}
for some $\ti c_0$ small enough.

We rewrite the second term as
\begin{align}
    &\sum_{j\in \A,1\leq j\leq N-1}2^{-j}\left[\left(-\iota_j\iota_{j+1}\right)\left(\frac{\lam_{j}(t)}{\lam_{j+1}(t)}\right)^k+\left(\iota_j\iota_{j-1}\right)\left(\frac{\lam_{j-1}(t)}{\lam_{j}(t)}\right)^k\right]\\
    =& \sum_{j\in \A,1\leq j\leq N-1}2^{-j-1}\left(\frac{\lam_{j}(t)}{\lam_{j+1}(t)}\right)^k\\
    &+\sum_{j=1}^{N-1} 2^{-j-1}\left(\frac{\lam_{j}(t)}{\lam_{j+1}(t)}\right)^k\left[-\iota_j\iota_{j+1}\mathbbm{1}_{\A}(j)+\iota_j\iota_{j+1}\mathbbm{1}_{\A}(j+1)\right],
\end{align}

which implies that
\begin{align}
    \psi'(t)\gtrsim\sum_{1\leq j\leq N-1,j\in \A}\left(\frac{\lam_{j}(t)}{\lam_{j+1}(t)}\right)^k-\ti c_0\bfd^2(t)
\end{align}
since $ -\iota_j\iota_{j+1}\mathbbm{1}_{\A}(j)+\iota_j\iota_{j+1}\mathbbm{1}_{\A}(j+1)\geq 0$ for $1\leq j\leq N-1$.

Define
\begin{align}\label{eq:phi}
    \phi(t)=\psi(t)-c_1\int_t^{T_+}\lam_N^2(s)\,\ud s-c_2\int_t^{T_+}|F(s)|\,\ud s,
\end{align}
where $c_1>0,c_2>0$ are chosen sufficiently large. Lemma~\ref{lem:energy-estimate-main}, together with the fact that $0<\ti c_0\ll 1$, gives
\begin{align}
    \phi'(t)\gtrsim\sum_{1\leq j\leq N-1,j\in \A}\left(\frac{\lam_{j}(t)}{\lam_{j+1}(t)}\right)^k-\ti c_0\bfd^2(t)+\lam_N^2(t)+|F(t)|\gtrsim \bfd^2(t),
\end{align}

which, using $\lim_{t\to T_+}\phi(t)=0$ and \eqref{eq:psi}, yields
\begin{align}\label{eq:d^2-phi}
    \int_{t}^{T_+}\bfd^2(s)\,\ud s\lesssim -\phi(t)\lesssim \lam_N(t)\bfd^{1+\frac{2}{k}}(t)+c_1\int_t^{T_+}\lam_N^2(s)\,\ud s+c_2\int_t^{T_+}|F(s)|\,\ud s.
\end{align}

From ~\eqref{eq:lam'},  we obtain 
\begin{align}
    \left|(\lam_N^2)'(t)|=2\cdot\lam_N(t)\right|\lam_N'(t)|\lesssim\bfd^2(t).
\end{align}

Therefore we have
\begin{align}
    \int_t^{T_+}\lam_N^2(s)\,\ud s\lesssim\int_{t}^{T_+}\int_{s}^{T_+}\bfd^2(\ta)\,\ud \ta\ud s
    &\lesssim (T_+-t)\int_t^{T_+}\bfd^2(s)\,\ud s.
\end{align}

For $F$,  ~\eqref{eq:F'-control-main} gives
\begin{align}
    \int_t^{T_+}|F(s)|\ud s\lesssim &\int_{t}^{T_+}\int_{s}^{T_+}(\bfd^2(\tau)+ \lam_N^2(\tau)|\log \lam_N(\tau)|)\ud \tau \ud s\\
    \lesssim & (T_+-t)\int_t^{T_+}\bfd^2(s)\,\ud s+(T_+-t)\int_t^{T_+} \lam_N^2(s)|\log \lam_N(s)|\ud s.
\end{align}

Therefore for $t$ close to $T_+$ we have
\begin{align}\label{eq:basic-control-d2}
    \int_{t}^{T_+}\bfd^2(s)\,\ud s\lesssim \lam_N(t)\bfd^{1+\frac{2}{k}}(t)+ (T_+-t)\int_t^{T_+}\lam_N^2(s)|\log \lam_N(s)|\ud s.
\end{align}
After a translation of time we assume that ~\eqref{eq:basic-control-d2} holds for $t\in [0,T_+)$.

Define
\begin{align}
    E=\left\{t\in [0,T_+):\lam_N(t)\bfd^{1+\frac{2}{k}}(t)> (T_+-t)\int_t^{T_+}\lam_N^2(s)(1+|\log \lam_N(s)|)\ud s\right\}.
\end{align}
It follows immediately that $E$ is open and hence measurable.

We define the maximal function
\begin{align}
    \lam_{N,M}(t)=\sup_{t\leq s<T_+}\lam_N(s),
\end{align}
then we have $0\leq -\lam'_{N,M}(t)\leq |\lam_N'(t)|$ for a.e. $t\in [0,T_+)$ since $\lam_N\in C^1$.

Claim: For $t\in [0,T_+)$ we have

\begin{align}\label{eq:control-E}
    \int_{[t,T_+)\cap E}\bfd (s)\ud s\lesssim o_{t\to T_+}(1)\lam_{N,M}(t), 
\end{align}
and for $t\in E^c$,
\begin{align}\label{eq:control-F}
    \int_{[t,T_+)}\bfd^2 (s)\ud s\lesssim (T_+-t)\int_{[t,T_+)}\lam_N^2(s)(1+|\log \lam_N(s)|)\ud s.
\end{align}

\emph{Proof of ~\eqref{eq:control-E}:}

For a.e. $t\in E$ we have
\EQ{\label{eq:E-d-control}
   & -\frac{\ud}{\ud t} \left[ \lam_{N,M}(t)^{\frac{k}{k+2}} \left( \int_t^{T_+} \bfd^2(s) \ud s \right)^{\frac{2}{k+2}} \right] \\
= & \frac{k}{k+2} \lam_{N,M}(t)^{-\frac{2}{k+2}} (-\lam_{N,M}'(t)) \left( \int_t^{T_+} \bfd^2(s) \ud s \right)^{\frac{2}{k+2}} + \frac{2}{k+2} \lam_{N,M}(t)^{\frac{k}{k+2}} \left( \int_t^{T_+} \bfd^2(s) \ud s \right)^{-\frac{k}{k+2}} \bfd^2(t)\\
\gtrsim & \lam_{N,M}(t)^{\frac{k}{k+2}}\left(\lam_{N,M}(t)\bfd^{1+\frac{2}{k}}(t)\right)^{-\frac{k}{k+2}}\bfd^2(t)\\
= & \bfd(t).
}

Since $\lam_{N,M}(t)^{\frac{k}{k+2}} \left( \int_t^{T_+} \bfd^2(s) \ud s\right)^{\frac{2}{k+2}} $ is decreasing in $[0,T_+)$,  ~\eqref{eq:E-d-control} implies, for $t\in  E$, that 
\begin{align}
    \int_{[t,T_+)\cap E}\bfd(s)\ud s\lesssim & \int_{[t,T_+)\cap E} -\frac{\ud}{\ud s} \left[ \lam_{N,M}(s)^{\frac{k}{k+2}} \left( \int_s^{T_+} \bfd^2(\tau) \ud\tau \right)^{\frac{2}{k+2}} \right] \ud s  \\
    \lesssim& \lam_{N,M}(t)^{\frac{k}{k+2}} \left( \int_t^{T_+} \bfd^2(s) \ud s \right)^{\frac{2}{k+2}}\\
    \lesssim&\lam_{N,M}(t)^{\frac{k}{k+2}} \left( \lam_{N,M}(t)\bfd^{1+\frac{2}{k}}(t) \right)^{\frac{2}{k+2}}\\
    =& o_{t\to T_+}(1)\lam_{N,M}(t) .
\end{align}

Now ~\eqref{eq:control-E} is proved for $t\in E$. For a general $t\in [0,T_+)$, we can take $s_n\in E$ such that
$s_n\downarrow \ti s_0:=\inf([t,T_+)\cap E)$. If such a sequence does not exist, we have $[t,T_+)\cap E=\varnothing$ and 
~\eqref{eq:control-E} holds automatically. Then we have
\begin{align}
    \int_{[t,T_+)\cap E}\bfd(s)\ud s= \lim_{n\to \infty}\int_{[s_n,T_+)\cap E}\bfd(s)\ud s.
\end{align}

Moreover, 
\begin{align}
    \int_{[s_n,T_+)\cap E}\bfd(s)\ud s\lesssim \lam_{N,M}(s_n)\bfd^{\frac{2}{k}}(s_n).
\end{align}
Passing to the limit gives
\begin{align}\int_{[t,T_+)\cap E}\bfd(s)\ud s= \lim_{n\to \infty}\int_{[s_n,T_+)\cap E}\bfd(s)\ud s\lesssim \lam_{N,M}(\ti s_0)\bfd^{\frac{2}{k}}(\ti s_0)\leq o_{t\to T_+}(1)\lam_{N,M}(t).
\end{align}
Hence, \eqref{eq:control-E} holds for all $t\in [0,T_+)$.

\emph{Proof of ~\eqref{eq:control-F}:} This follows immediately from ~\eqref{eq:basic-control-d2} and the definition of $E$.

Now we use ~\eqref{eq:control-E} and ~\eqref{eq:control-F} to derive a contradiction. 

We define
\begin{align}
    t_n:=\inf \left\{t\in[0,T_+):\lam_{N,M}(t)\leq \frac{1}{2^n}\right\}.
\end{align}
Then for $n$ sufficiently large $t_n\in (0,T_+)$. We also have that $t_n$ is increasing to $T_+$ as $n\to \infty$.

On the time interval $[t_n,t_{n+1}]$, ~\eqref{eq:control-E}  gives 
\begin{align}
    \frac{1}{2^{n+1}}\leq & \int_{t_n}^{t_{n+1}}-\lam_{N,M}'(t)\ud t\\
    \leq & \int_{[t_n,t_{n+1}]\cap E}-\lam_{N,M}'(t)\ud t+\int_{[t_n,t_{n+1}]\cap E^c}-\lam_{N,M}'(t)\ud t\\
    \lesssim & \int_{[t_n,t_{n+1}]\cap E}\bfd(t)\ud t+\int_{[t_n,t_{n+1}]\cap E^c}\bfd(t)\ud t\\
    \leq & o_{n\to \infty}(1) \lam_{N,M}(t_n)+\int_{[t_n,t_{n+1}]\cap E^c}\bfd(t)\ud t,
\end{align}
where the passage from the second line to the third follows from the estimate $ -\lam_{N,M}'(t)\lesssim \bfd(t) $ with an implicit constant independent of $n$. This estimate is an immediate consequence of the definition of $\lam_{N,M}(t)$, ~\eqref{eq:lam'}, and the definition of $\bfd(t)$. 

Since $\lam_{N,M}(t_n)=\frac{1}{2^n}$, we have for $n$ sufficiently large,
\begin{align}\label{eq:tn-d}
    \frac{1}{2^{n+2}}\leq \int_{[t_n,t_{n+1}]\cap E^c}\bfd(t)\ud t.
\end{align}

Define $\tau_n=\inf\{t:t\in[t_n,t_{n+1}]\cap E^c\}$. By ~\eqref{eq:tn-d}, and since $E$ is open, we know that $\tau_n\in [t_n,t_{n+1}]\cap E^c$ is well-defined. Combining ~\eqref{eq:control-F} with the fact that $x^2|\log x|$ is increasing when $0<x\ll 1$, we obtain
\begin{align}
    \int_{[t_n,t_{n+1}]\cap E^c}\bfd(t)\ud t\lesssim & \left(t_{n+1}-t_n\right)^{\frac{1}{2}}\left(\int_{[t_n,t_{n+1}]\cap E^c}\bfd^2(t)\ud t\right)^{\frac{1}{2}}\\
    \lesssim & \left(t_{n+1}-t_n\right)^{\frac{1}{2}}\left(\int_{[\tau_n,T_+)}\bfd^2(t)\ud t\right)^{\frac{1}{2}}\\
    \lesssim  & \left(t_{n+1}-t_n\right)^{\frac{1}{2}}(T_+-\ta_n)^{\frac{1}{2}}\left(\int_{[\tau_n,T_+)}\lam_N^2(s)(1+|\log \lam_N(s)|)\ud s\right)^{\frac{1}{2}}\\
    \lesssim & \left(t_{n+1}-t_n\right)^{\frac{1}{2}}(T_+-t_n)^{\frac{1}{2}}\left(\int_{[t_n,T_+)}\lam_N^2(s)(1+|\log \lam_N(s)|)\ud s\right)^{\frac{1}{2}}\\
    \lesssim & (t_{n+1}-t_n)^{\frac{1}{2}}(T_+-t_n)\lam_{N,M}(t_n)(|\log \lam_{N,M}(t_n)|)^{\frac{1}{2}}\\
    =& (t_{n+1}-t_n)^{\frac{1}{2}}(T_+-t_n) \frac{1}{2^n}n^{\frac{1}{2}}
\end{align}

Thus,
\begin{align}\label{eq:tn-tn+1}
    \frac{1}{2^n}\lesssim (t_{n+1}-t_n)^{\frac{1}{2}}(T_+-t_n) \frac{1}{2^n}n^{\frac{1}{2}}.
\end{align}

Setting $S_n=T_+-t_n$, we can rearrange ~\eqref{eq:tn-tn+1} into
\begin{align}
    \frac{1}{n}\lesssim (S_{n}-S_{n+1})S_n^2.
\end{align}

Taking the sum from $n=n_0\gg 1$ to $\infty$, we have
\begin{align}
    \infty=\sum_{n\geq n_0}\frac{1}{n}\lesssim \sum_{n\geq n_0}(S_{n}-S_{n+1})S_n^2\leq \sum_{n\geq n_0}(S_n^3-S_{n+1}^3)=S_{n_0}^3<\infty,
\end{align}
which leads to a contradiction. Therefore, we have proved Theorem~\ref{thm:main}.

\appendix

\section{Estimates for equivariant functions}
\label{sec:appendix-f}

In this section we prove ~\eqref{eq:f-k-1} and ~\eqref{eq:f-k-2}. We will omit the time variable $t$ for convenience and denote $f'(r)=\p_r f(r)$. It is easy to check that all implicit constants in this section are independent of $t$.

In the rest of this section we let $(x,y)=re^{i\te}\in \R^2$; a basic calculation yields
    \begin{align}
        \p_x=\cos \te \p_r-\frac{\sin \te}{r}\p_{\te},\ \p_y=\sin \te\p_r+\frac{\cos \te}{r}\p_{\te}.
    \end{align}
    We let $D_{+}=\p_x+i\p_y=e^{i\te}\left(\p_r+\frac{i}{r}\p_{\te}\right)$ and $D_-=\p_x-i\p_y=e^{-i\te}\left(\p_r-\frac{i}{r}\p_{\te}\right)$.

    For a general function $G(r,\te)=g(r)e^{im\te}$ with $m\in \Z$, we have
    \begin{align}
        D_+ G(r,\te)=\left(g'(r)-m\frac{g(r)}{r}\right)e^{i(m+1)\te},\ D_- G(r,\te)=\left(g'(r)+m\frac{g(r)}{r}\right)e^{i(m-1)\te},
    \end{align}
    \begin{align}\label{eq:nabla}
        |\na G|^2(r,\te)=|g'(r)|^2+\frac{m^2}{r^2}|g(r)|^2,
    \end{align}
    and
    \begin{align}
\label{eq:Hessian}|\na^2G|^2(r,\te)=|g''(r)|^2+2m^2\left|\p_r\left(\frac{g(r)}{r}\right)\right|^2+\left|\frac{g'(r)}{r}-m^2\frac{g(r)}{r^2}\right|^2.
    \end{align}
    The first equality can be computed directly. For the last equality, we consider the orthogonal frame $\{\p_r,\frac{1}{r}\p_{\te}\}$ and compute using 
    \begin{align}
        |\na^2G|^2(r,\te)=\left|\na^2 G(\p_r,\p_r)\right|^2(r,\te)+2\left|\na^2 G(\p_r,\frac{1}{r}\p_{\te})\right|^2(r,\te)+\left|\na^2 G(\frac{1}{r}\p_{\te},\frac{1}{r}\p_{\te})\right|^2(r,\te),
    \end{align}
    where $\na^2$ is the Hessian defined in $\R^2$ with Euclidean metric. Combining ~\eqref{eq:nabla} with ~\eqref{eq:Hessian} we get for $0<r_0\leq \infty$,
    \begin{align}\label{eq:hardy-angle}
        m^2\int_0^{r_0}\frac{|g(r)|^2}{r^2}r\ud r+m^2\int_0^{r_0} \left|\p_r\left(\frac{g(r)}{r}\right)\right|^2 r\ud r\lesssim\|g(r)e^{im\te}\|_{H^2(B(0,r_0))}^2
    \end{align}

    We first consider ~\eqref{eq:f-k-1}. Using ~\eqref{eq:hardy-angle} with $r_0=1$ for $G(r)=D_+\left(f(r)e^{ik\te}\right)=\left(f'(r)-\frac{k}{r}f(r)\right)e^{i(k+1)\te}$ and $G(r)= D_-\left(f(r)e^{ik\te}\right)=\left(f'(r)+\frac{k}{r}f(r)\right)e^{i(k-1)\te}$ respectively, we obtain
    \begin{align}
        &\int_0^{1}\frac{|f'(r)\mp\frac{k}{r}f(r)|^2}{r^2}r\ud r+(k\pm1)^2\int_0^{1} \left|\p_r\left(\frac{f'(r)\mp\frac{k}{r}f(r)}{r}\right)\right|^2 r\ud r\\
    \lesssim&\|D_{\pm}\left(f(r)e^{ik\te}\right)\|_{H^2(B(0,1))}^2\\
    \lesssim &\|f(r)e^{ik\te} \|_{H^3(B(0,1))}^2.
    \end{align}
    
    Then ~\eqref{eq:f-k-1} follows from $k\geq 2$ and
    \begin{align}
        \frac{k}{r^2}f(r)= \frac{1}{2}\left[\left( \frac{f'(r)}{r}+ \frac{k}{r^2}f(r) \right)- \left( \frac{f'(r)}{r}- \frac{k}{r^2}f(r) \right)\right].
    \end{align}

    Next we consider ~\eqref{eq:f-k-2}. We let $F(x,y)=f(r)e^{ik\te}$. A direct computation yields
    \begin{align}
        \frac{f(r)}{r^3}e^{ik\theta}
=&\frac{1}{ik(k^2-1)(k^2-4)r^3}
\Big[
\big((k^2-1)y^3-3x^2y-3ikxy^2\big)\partial_x^3F \\
&\quad
+3\big(x^3-(k^2+1)xy^2+ik\,y(2x^2-y^2)\big)
\partial_x^2\partial_yF \\
&\quad
+3\big((k^2+1)x^2y-y^3-ik\,x(x^2-2y^2)\big)
\partial_x\partial_y^2F \\
&\quad
+\big(3xy^2-3ikx^2y-(k^2-1)x^3\big)\partial_y^3F
\Big].
    \end{align}

    This combined with $\frac{|x|}{r},\frac{|y|}{r}\leq 1$ and $k\geq 3$ yields
    \begin{align}
        \int_0^1\left|\frac{f(r)}{r^3}\right|^2r\ud r\lesssim \|f(r)e^{ik\te}\|_{H^3(B(0,1))}^2\lesssim1.
    \end{align}
    Hence ~\eqref{eq:f-k-2} is proved.

\section{Estimates on integrals for $F_j$}
\label{sec:appendix-Fj}

In this section we prove some estimates that appeared in Section~\ref{Sec:4}. For convenience we omit the time variable $t$ and note that our estimates are uniform in $t$. By taking $t$ close enough to $T_+$, we may assume that
\begin{align}
    0<\lam_1<\cdots<\lam_j<\cdots<\lam_N\ll1.
\end{align}

\begin{lem}\label{lem:appendix}

    For $1\leq i<j\leq N$, we have

    \begin{align}\label{eq:appendix-1}
        \int_0^{\infty}F_0^2(r)F_j(r)\frac{\ud r}{r}\lesssim \lam_j^2.
    \end{align}

    \begin{align}\label{eq:appendix-2}
        \int_0^{\infty}F_0^2(r)F_i(r)F_j(r)\frac{\ud r}{r}\lesssim \begin{cases}
\lambda_i^2 \lambda_j^2 (1 + |\log \lambda_j|)^2, & k = 2, \\
\lambda_j^4 \left( \frac{\lambda_i}{\lambda_j} \right)^k, & k \ge 3.
\end{cases},
    \end{align}

    \begin{align}\label{eq:appendix-3}
        \int_0^{\infty}F_0^2(r)F_j^2(r)\frac{\ud r}{r}\lesssim \begin{cases} \lambda_j^4 (1 + |\log \lambda_j|)^2, & k = 2, \\ \lam_j^4,& k\geq 3 .\end{cases}
    \end{align}

    \begin{align}\label{eq:appendix-4}
        \int_0^{\infty}F_0(r)F_i(r)F_j(r)\frac{\ud r}{r}\lesssim \begin{cases}
\lambda_i^2 (1 + |\log \lambda_j|)^{\frac{1}{2}}, & k = 2, \\
\lambda_j^2 \left( \frac{\lambda_i}{\lambda_j} \right)^k, & k \ge 3.
\end{cases}
    \end{align}

    \begin{align}\label{eq:appendix-5}
        \int_0^{\infty}\ti F_0(r)F_j(r)\frac{\ud r}{r}\lesssim \lambda_j.
    \end{align}

    \begin{align}\label{eq:appendix-6}
        \int_0^{\infty}\ti F_0(r)F_0(r)F_j(r)\frac{\ud r}{r}\lesssim \lam_j^2.
    \end{align}

    \begin{align}\label{eq:appendix-7}
        \int_0^{\infty}\ti F_0(r)F_i(r)F_j(r)\frac{\ud r}{r}\lesssim \lambda_j \left( \frac{\lambda_i}{\lambda_j} \right)^k.
    \end{align}

    \begin{align}\label{eq:appendix-8}
        \int_0^{\infty}F_0(r)F_j^3(r)\frac{\ud r}{r}\lesssim\begin{cases}
\lambda_j^2 (1 + |\log \lambda_j|)^{\frac{1}{2}}, & k = 2, \\
\lambda_j^2, & k \ge 3.
\end{cases}.
    \end{align}
    
    \begin{align}\label{eq:appendix-9}
        \int_0^{\infty}F_0^3(r)F_j(r)\frac{\ud r}{r}\lesssim \lam_j^2.
    \end{align}

\end{lem}

\begin{proof}

    We record four fundamental inequalities, from which ~\eqref{eq:appendix-1}--~\eqref{eq:appendix-9} can be obtained easily.  For $0<\lam<1$, $m>0$ and $\al\geq 0$,
    \begin{align}\label{eq:B-1}
        \int_0^\lam r^{m-1}(1+|\log r|)^\al\ud r \lesssim \lam^m(1+|\log \lam|)^\al,
    \end{align}
    \begin{align}\label{eq:B-2}
        \int_\lam^1 r^{-m-1}(1+|\log r|)^\al \ud r \lesssim \lam^{-m}(1+|\log \lam|)^\al.
    \end{align}
    \begin{align}\label{eq:B-3}
        \int_\lam^1 r^{-1}(1+|\log r|)^\al \ud r \lesssim (1+|\log \lam|)^{\al+1}.
    \end{align}
    \begin{align}\label{eq:B-4}
        \int_\lam^1 r^{m-1}(1+|\log r|)^\al\,\ud r
    &\lesssim 1.
    \end{align}
    The proof of ~\eqref{eq:B-1}, ~\eqref{eq:B-2}, ~\eqref{eq:B-3} and ~\eqref{eq:B-4} is easy:
    \begin{align}
        \int_0^{\lam} r^{m-1}(1+|\log r|)^\al\ud r  =& \lam^m \int_0^1 t^{m-1} (1 - \log \lam - \log t)^\al \ud t\\
        \lesssim & \lam^m\int_0^1 t^{m-1}\left( (1-\log \lam)^{\al}+|\log t|^{\al} \right)\ud t\\
        \lesssim & \lam^m(1+|\log \lam|)^\al,
    \end{align}
    \begin{align}
        \int_{\lam}^1 r^{-m-1}(1+|\log r|)^\al \ud r = & \lam^{-m} \int_1^{\frac{1}{\lam}} t^{-m-1} (1 - \log \lam - \log t)^\al \ud t\\
        \leq & \lam^{-m} \int_1^{\frac{1}{\lam}} t^{-m-1}(1-\log \lam)^{\al} \ud t \\
        \lesssim & \lam^{-m}(1+|\log \lam|)^\al,
    \end{align}
    \begin{align}
        \int_{\lam}^1 r^{-1}(1+|\log r|)^\al \ud r = & \ \int_1^{\frac{1}{\lam}}r^{-1}(1-\log r-\log \lam )^{\al} \ud r\\
        \leq & \ \int_1^{\frac{1}{\lam}}r^{-1}(1-\log \lam)^{\al} \ud r\\
        =& (1-\log \lam)^{\al}\log\frac{1}{\lam}\\
        \leq & (1+|\log \lam|)^{\al+1}.
    \end{align}
    \begin{align}
        \int_{\lam}^1 r^{m-1}(1+|\log r|)^\al\,\ud r
    \lesssim  \int_0^1 r^{m-1}(1+|\log r|)^\al\,\ud r \lesssim 1.
    \end{align}

    We note that the implicit constants in ~\eqref{eq:B-1} and ~\eqref{eq:B-2}
depend on $m$ and $\alpha$. However, for each fixed $k$, only finitely many
pairs $(m,\al)$ occur in the estimates below. Hence, after enlarging the
constant if necessary, these inequalities can be applied uniformly throughout
the proof, with constants depending only on $k$.

    We will prove ~\eqref{eq:appendix-4}, the proofs of the remaining inequalities are very similar by splitting intervals. 
    
    When $k=2$, applying \eqref{eq:B-1} with
$(m,\al)=(6,\frac{1}{2})$ for the integral on $(0,\lam_i)$,
\eqref{eq:B-1} with $(m,\al)=(2,\frac{1}{2})$ for the integral on
$(\lam_i,\lam_j)$ after extending it to $(0,\lam_j)$, and
\eqref{eq:B-2} with $(m,\al)=(2,\frac{1}{2})$ , we obtain
    \begin{align}
        \int_0^{\infty}F_0(r)F_i(r)F_j(r)\frac{\ud r}{r}\lesssim &\ \int_0^{\lam_i}r^2(1+|\log r|)^{\frac{1}{2}}\left(\frac{r}{\lam_i}\right)^2\left(\frac{r}{\lam_j}\right)^2\frac{\ud r}{r} \\
        & +\int_{\lam_i}^{\lam_j}r^2(1+|\log r|)^{\frac{1}{2}}\left(\frac{\lam_i}{r}\right)^2\left(\frac{r}{\lam_j}\right)^2\frac{\ud r}{r}\\
        & + \int_{\lam_j}^1 r^2(1+|\log r|)^{\frac{1}{2}}\left(\frac{\lam_i}{r}\right)^2\left(\frac{\lam_j}{r}\right)^2\frac{\ud r}{r}\\
        & + \int_1^{\infty} \left(\frac{\lam_i}{r}\right)^2\left(\frac{\lam_j}{r}\right)^2\frac{\ud r}{r}\\
        \lesssim & \frac{1}{\lam_i^2\lam_j^2}\lam_i^6(1+|\log \lam_i|)^{\frac{1}{2}}+\int_{0}^{\lam_j}r^2(1+|\log r|)^{\frac{1}{2}}\left(\frac{\lam_i}{r}\right)^2\left(\frac{r}{\lam_j}\right)^2\frac{\ud r}{r}\\
        &+\lam_i^2\lam_j^2 \lam_j^{-2}(1+|\log\lam_j|)^{\frac{1}{2}}+\lam_i^2\lam_j^2\\
        \lesssim &  \frac{\lam_i^2}{\lam_j^2}\lam_i^2(1+|\log\lam_i|)^{\frac{1}{2}}+\frac{\lam_i^2}{\lam_j^2}\lam_j^2(1+|\log\lam_j|)^{\frac{1}{2}}\\
        &+\lam_i^2\lam_j^2 \lam_j^{-2}(1+|\log\lam_j|)^{\frac{1}{2}}+\lam_i^2\lam_j^2\\
        \lesssim & \lam_i^2(1+|\log\lam_j|)^{\frac{1}{2}}.
    \end{align}
    In the last line we use the fact that $x^2|\log x|$ is increasing in $(0,e^{-\frac{1}{2}})$.
    
    When $k\geq3$ we have
    \begin{align}
        \int_0^{\infty}F_0(r)F_i(r)F_j(r)\frac{\ud r}{r}\lesssim &\ \int_0^{\lam_i}r^2\left(\frac{r}{\lam_i}\right)^k\left(\frac{r}{\lam_j}\right)^k\frac{\ud r}{r}  +\int_{\lam_i}^{\lam_j}r^2\left(\frac{\lam_i}{r}\right)^k\left(\frac{r}{\lam_j}\right)^k\frac{\ud r}{r}\\
        & + \int_{\lam_j}^1 r^2\left(\frac{\lam_i}{r}\right)^k\left(\frac{\lam_j}{r}\right)^k\frac{\ud r}{r}+ \int_1^{\infty} \left(\frac{\lam_i}{r}\right)^k\left(\frac{\lam_j}{r}\right)^k\frac{\ud r}{r}\\
        \lesssim &\ \lam_j^2\left(\frac{\lam_i}{\lam_j} \right)^k.
    \end{align}

\end{proof}

\bibliographystyle{plain}
\bibliography{WM}

\end{document}